\documentclass[final]{siamltex}

\def\ds{\displaystyle}

\newcommand{\bea}{\begin{eqnarray}}
\newcommand{\eea}{\end{eqnarray}}
\newcommand{\beas}{\begin{eqnarray*}}
\newcommand{\eeas}{\end{eqnarray*}}
\newcommand{\beq}{\begin{equation}}
\newcommand{\eeq}{\end{equation}}

\title{An ADI Crank-Nicolson Orthogonal Spline Collocation Method for the Two-Dimensional Fractional Diffusion-Wave Equation\thanks{This research was supported in part by the National Nature Science Foundation of China (contract grant 11271123) and the Research and Innovation Project for College Graduates of Hunan Province (contract grant CX2012B196).}}

\author{Graeme Fairweather\thanks{Corresponding author. Current address: Mathematical Reviews, American Mathematical Society, 416 Fourth Street, Ann Arbor, MI 48103, USA.({\tt gxf@ams.org}).}
\and Xuehua Yang\thanks{College of Mathematics and Computer Science, Key Laboratory of High Performance Computing and Stochastic Information Processing (Ministry of Education of China), Hunan Normal University, Changsha, Hunan 410081, P. R. China ({\tt hunanshidayang@163.com}).}
  \and Da Xu\thanks{College of Mathematics and Computer Science, Hunan Normal University, Changsha, Hunan 410081, P. R. China ({\tt daxu@hunnu.edu.cn}).}
\and Haixiang Zhang\thanks{School of Mathematics, Central South University, Changsha, Hunan 410075, P. R. China ({\tt hassenzhang@163.com}).}}

\begin{document}

\maketitle

\begin{abstract}
A new method is formulated and analyzed for the approximate solution of a two-dimensional time-fractional diffusion-wave equation. In this method, orthogonal spline collocation is used for the spatial discretization and, for the time-stepping, a novel alternating direction implicit (ADI) method based on the Crank-Nicolson method combined with the $L1$-approximation of the time Caputo derivative of order $\alpha\in(1,2)$. It is proved that this scheme is stable, and of optimal accuracy in various norms. Numerical experiments demonstrate the predicted global convergence rates and also superconvergence.
\end{abstract}



\pagestyle{myheadings}
\thispagestyle{plain}
\markboth{GRAEME FAIRWEATHER, XUEHUA YANG, DA XU AND HAIXIANG ZHANG}{AN ADI OSC METHOD FOR THE FRACTIONAL DIFFUSION-WAVE EQUATION}

\section{Introduction}
In this paper, we focus on the formulation and analysis of an alternating direction implicit (ADI) orthogonal spline collocation (OSC) method for the approximate solution of the two-dimensional time-fractional diffusion-wave problem
\begin{equation}\label{eq:1}
_0^CD_t^{\alpha}u(x,y,t)=\Delta u(x,y,t)+f(x,y,t), \quad (x,y,t)\in \Omega_T\equiv\Omega\times(0,T],
\end{equation}
 with the initial conditions
\begin{eqnarray}\label{eq:2}
u(x,y,0)=\varphi(x,y),\quad D_t u(x,y,0)=\phi(x,y),\quad  (x,y)\in \overline{\Omega}=\Omega\cup\partial\Omega,
\end{eqnarray}
and the boundary condition
\begin{eqnarray}\label{eq:3}
u(x,y,t)=0,\quad (x,y,t)\in \partial\Omega\times(0,T].
\end{eqnarray}
Here, $\Delta$ is the Laplace operator,
$\Omega=(0,1)\times(0,1)$ with boundary $\partial \Omega$,
$\varphi(x,y),\phi(x,y)$ and $f(x,y,t)$ are
given sufficiently smooth functions in their respective domains. Also, $_0^CD_t^{\alpha}u(x,y,t)$ is the Caputo fractional derivative of order $\alpha$ $(1<\alpha<2)$ defined by
\begin{equation}\label{eq:4}
_0^CD_t^{\alpha}u(x,y,t)=\frac{1}{\Gamma(2-\alpha)}\int_0^t\frac{\partial^2 u(x,y,s)}{\partial s^2}\frac{ds}{(t-s)^{\alpha-1}},
\end{equation}
where $\Gamma(\cdot)$ denotes the Gamma function, named after Caputo \cite{Ca} who was one of the first to use this operator in applications and to investigate some of its properties. Equation (\ref{eq:1}) is called a time-fractional partial differential equation of order $\alpha$ since it is intermediate between the diffusion equation ($\alpha = 1$)
and the wave equation ($\alpha = 2$).  In recent years, fractional partial differential equations have gained rapidly in popularity and importance as new modeling tools in a variety of fields, such as physics, biology, mechanical engineering, environmental science, signal processing, systems identification, electrical and control theory, finance, and hydrology; see, for example, \cite{Uc1,Uc2}. In particular, the fractional diffusion-wave equation (\ref{eq:1}) models wave propagation in viscoelastic materials.

Several approaches have been proposed for the solution of fractional partial differential equations in one and several space variables; see, for example, \cite{BaDiScTr,cuimingrong2,JiLaZh,MuMc,ReSu}
and references in these papers. In particular, alternating direction implicit (ADI) methods have been employed recently for the solution of multidimensional problems.
ADI methods were first introduced in the context of finite difference methods (FDMs) for parabolic and elliptic problems by Peaceman and Rachford \cite{ADIqiyuan2} in the 1950s, and such methods in conjunction with various types of spatial discretizations continue to be studied extensively today, especially for the numerical solution of time-dependent problems; see \cite{FeFajcp} and references therein.
The attraction of these techniques is that they replace the solution of multidimensional problems by sequences of one-dimensional problems, thus  reducing the computational cost.
For solving fractional problems in two space variables, ADI methods have been employed in numerous contexts. Meerschaert et al.,  \cite{M3eScTa} formulated an ADI FDM based on the backward Euler method to solve a class of space-fractional partial differential
equations with variable coefficients, and, for the same problem, Tadjeran and Meerschaert \cite{cuiminghou32}
derived an ADI method based on the Crank-Nicolson finite difference method, and used Richardson extrapolation to improve the spatial accuracy.
For a space-fractional advection-dispersion equation, Chen and Liu \cite{liufawangADI1} considered an ADI FDM backward Euler method and obtained second-order accuracy in both space and time on using Richardson extrapolation.
Zhang and Sun \cite{sunzhizhong7} formulated and analyzed two ADI FDMs based on the $L1$ approximation \cite{OlSp} and the backward Euler method for the time-fractional sub-diffusion equation comprising (\ref{eq:1})--(\ref{eq:3}) with $0 < \alpha <1$.  These methods are proved to be second-order in space and of order $\min(2\alpha,2-\alpha)$ and $\min(1+\alpha,2-\alpha)$, respectively, in time.
Zhang et al., \cite{sunzhizhong2012} formulated and analyzed a compact ADI FDM and a Crank-Nicolson ADI FDM for the time-fractional diffusion-wave equation and proved that the methods are fourth-order accurate in space and of order ${3-\alpha}$ in time. For the same problem but with Neumann boundary conditions,
Ren and Sun \cite{ReSu} formulated similar ADI methods of the same accuracy.
Wang and Wang \cite{cuiminghou35} formulated an ADI FDM for a class of space-fractional diffusion equations. They provided no analysis of the method but demonstrated its efficiency.
Cui considered compact ADI FDMs for a time-fractional diffusion equation with the Riemann-Liouville fractional derivative in \cite{cuimingrong1} and the Caputo derivative
in \cite{cuimingrong2}.
ADI FDMs have also been used in the solution of three-dimensional fractional problems. In particular,
Liu et al., \cite{liufawangADI2} proposed such a
scheme for the solution of a fractional equation governing seepage flow, and used Richardson extrapolation to improve the spatial accuracy. Also, Yu et al., \cite{liufawangADI3} constructed an ADI FDM method for the fractional Bloch-Torrey equation to study anomalous diffusion in the human brain.

Orthogonal spline collocation has evolved as a valuable technique for the solution of several types of partial differential
equations \cite{BiFa}, especially in combination with ADI methods for multidimensional problems \cite{FeBiFa}. The popularity of OSC methods is due in part to their conceptual
simplicity, wide applicability and ease of implementation. A well-known advantage of OSC methods over finite element Galerkin methods is that the calculation of the coefficients in the equations determining the approximate solution is very fast, since no integrals need to be evaluated
or approximated. Another attractive feature of OSC methods is their superconvergence properties; see, for example, \cite{fairweather2010}.

A brief outline of the remainder of this paper is as follows. In section 2, standard
notation and basic lemmas
are presented.
The ADI OSC Crank-Nicolson method for the solution of problem (\ref{eq:1})--(\ref{eq:3}) is formulated in section 3, followed by a stability analysis of the scheme in section 4. In section 5, we derive error estimates in the $H^{\ell}$ norm, $\ell=0,1,2$, at each time step. In section 6, we present the results of numerical
experiments which support the
analytical rates of convergence and exhibit superconvergence. Some concluding remarks are provided in section 7.

\section{Preliminaries}
In this section, we introduce standard notation used in the formulation of OSC methods, and basic lemmas used in their analysis.

For positive integers $r$, $N_x$, $N_y$, let
$\delta_x=\{x_i\}_{i=0}^{N_x}$ and $\delta_y=\{y_j\}_{j=0}^{N_y}$ be two partitions of
$\overline{I}=[0,1]$ such that
\[
0=x_0<x_1<...<x_{N_x}=1,\qquad 0=y_0<y_1<...<y_{N_y}=1.
\]
Set
\[
\begin{array}{lll}I_k^x=(x_{k-1},x_{k}), &h_k^x=x_{k}-x_{k-1}, &1\leq k\leq N_x,\\\\
I_l^y=(y_{l-1},y_{l}),&h_l^y=y_{l}-y_{l-1}, &1\leq l\leq N_y,
\end{array}
\]
and 
$h=\max\left(\max \limits_{1\leq k\leq N_x} h_k^x,\max \limits_{1\leq l\leq N_y} h_l^y\right)$.
It is assumed that the collection of partitions
$\delta=\delta_x\times\delta_y$ of $\Omega$ is quasi-uniform.

Let $\mathcal{M}(r,\delta_x)$ and $\mathcal{M}(r,\delta_y)$ be the spaces of piecewise polynomials of degree $\leq r$, $r\geq 3$, defined by
$$\mathcal{M}(r,\delta_x)=\left\{v|v\in C^1(\bar{I}),v|_{\overline{I}^x_k}\in
P_r, k=1, 2 ,..., N_x, v(0)=v(1)=0\right\},$$
and
$$\mathcal{M}(r,\delta_y)=\left\{v|v\in C^1(\bar{I}),v|_{\overline{I}^y_l}\in
P_r, l=1, 2 ,..., N_y, v(0)=v(1)=0\right\},$$
where $P_r$ denotes the set of polynomials of degree at most $r$. Then we set
\[
\mathcal{M}(\delta)=\mathcal{M}(r,\delta_x)\otimes \mathcal{M}(r,\delta_y),
\]
the set of all functions that are finite linear combinations of products $v^x(x)v^y(y)$, where $v^x\in\mathcal{M}(r,\delta_x)$ and $v^y\in\mathcal{M}(r,\delta_y)$, and $\mbox{dim}\,\mathcal{M}(\delta)=(r-1)^2N_xN_y.$

Let $\{\lambda_k\}_{k=1}^{r-1}$, with $0<\lambda_1<\lambda_2<\cdots<\lambda_{r-1}<1$, denote the nodes of the
$(r-1)$-point Gauss quadrature rule on the interval
$\overline{I}$ with corresponding weights
$\{\omega_k\}_{k=1}^{r-1}$, and let $\Lambda_x=\{\xi_{i,k}^x\}_{i,k=1}^{N_x,r-1}$ and  $\Lambda_y=\{\xi_{j,l}^y\}_{j,l=1}^{N_y,r-1}$ be the sets of Gauss points in the $x$- and $y$-directions, respectively, where
\[
\xi_{i,k}^x=x_{i-1}+\lambda_kh^x_i,\quad 1\leq
k\leq r-1,\quad 1\leq i\leq N_x,\]
and
\[
\xi_{j,l}^y=y_{j-1}+\lambda_lh_j^y, \quad 1\leq
l\leq r-1, \quad 1\leq j\leq N_y.
\]
Then
$\Lambda=\{\xi|\xi=(\xi^x, \xi^y), \xi^x\in\Lambda_x, \xi^y\in\Lambda_y\}$ is the set of
Gauss quadrature points in $\Omega$, which are the collocation points.

For $u$ and $v$ defined on $\Lambda$, we define the discrete inner product
$\left \langle\cdot,\cdot\right \rangle$ and norm $\|\cdot\|_D$ by
\[
\left\langle
u,v\right\rangle=\sum\limits_{i=1}^{N_x}\sum\limits_{j=1}^{N_y}
h_i^xh_j^y\sum\limits_{k=1}^{r-1}\sum\limits_{l=1}^{r-1}\omega_k
\omega_l(uv)(\xi_{i,k}^x,\xi_{j,l}^y),\qquad  \|v\|_D^2=\langle
v,v\rangle.
\]

For $\ell$ a nonnegative integer, we denote by
\begin{equation}
\|f\|_{H^{\ell}}=\left(\sum\limits_{0\leq\alpha_1+\alpha_2\leq\ell}\left\|
\frac{\partial^{\alpha_1+\alpha_2}f}{\partial x^{\alpha_1}\partial y^{\alpha_2}}\right\|^2\right)^{\frac{1}{2}}
\end{equation}
the norm on the Sobolev space $H^{\ell}(\Omega)$, where $\| \cdot \|$ denotes the usual $L^2$ norm, sometimes written as $\|\cdot\|_{H^0}$ for convenience.

If $X$ is a normed space with norm $\|\cdot\|_X$, then we denote by $C\left([0,T],X\right)$ the set of functions $f\in C(\overline{\Omega}_T)\equiv C^{0, 0,0}(\overline{\Omega}_T)$ such that $f(\cdot,t)\in X$ for $t\in[0,T]$, and
$$\|f\|_{C([0,T],X)}=\max \limits_{0\leq t\leq T} \|f(\cdot,t)\|_X<\infty.$$

Let $C^{p,q,s}(\overline{\Omega}_T)$ denote the set of functions $f$ such that
$\displaystyle\frac{\partial^{i+j+n}f}{\partial x^i\partial y^j\partial t^n}$ is  continuous on $\overline{\Omega}_T$ for $0\leq i\leq p$, $0\leq j\leq q$, and $0\leq n\leq s$. If $f\in C^{p,q,s}(\overline{\Omega}_T)$, then $\|f\|_{C^{p,q,s}}$ is defined by
$$\|f\|_{C^{p,q,s}}=\max \limits_{ 0\leq i\leq p, 0\leq j\leq q, 0\leq n\leq s}\max \limits_{(x,y,t)\in\overline{\Omega}_T}\left |\frac{\partial^{i+j+n}f}{\partial x^i\partial y^j\partial t^n}\right|.$$

Throughout the paper, we denote by $C$ a generic positive constant that is independent of $h$ and $\Delta t$, unless otherwise noted and
is not necessarily the same on each occurrence. Besides, we make repeated use of the Young's inequality
\begin{eqnarray}
\label{eqyoung}
de\leq\varepsilon d^2+\frac{1}{4\varepsilon}e^2,\quad d, e \in {\cal R}, \quad \varepsilon>0,
\end{eqnarray}
Next we present several lemmas required in the stability and convergence analyses.
\begin{lemma}\label{lem:2.2}
If $U,V\in\mathcal{M}(\delta)$, then the following hold:
\begin{eqnarray}\label{eq2.3}
\left \langle-\Delta U,V\right \rangle=\left \langle U,-\Delta V\right \rangle,
\end{eqnarray}
{\rm \cite[Eq. (3.4)]{FeFa}};

\begin{eqnarray}\label{eq2.4}
\left \langle-\Delta U,U\right \rangle \; \geq \; C\left\|\nabla U\right\|^2\; \geq \; 0,
\end{eqnarray}
{\rm \cite[Eq. (3.5)]{FeFa}};

\begin{eqnarray}\label{eq2.5}
\left|\left \langle\Delta U,V\right \rangle\right| \; \leq \; C\left\| \nabla U\right\|\left \|\nabla V\right\|, \quad -\left \langle\Delta U,V\right \rangle \; \leq \; C\left[\left\| \nabla U\right\|^2+\left \|\nabla V\right\|^2\right];
\end{eqnarray}
{\rm see the proof of Lemma 3.3 in \cite{FeFa}}.
\end{lemma}

\begin{lemma}
\label{lem:2.1}
For $V\in\mathcal{M}(\delta)$,
\begin{eqnarray}\label{eq2.2}
\left \langle\frac{\partial^4V}{\partial x^2\partial y^2},V\right \rangle \; \geq \; \left\| \frac{\partial^2V}{\partial x\partial y}\right \|^2,
\end{eqnarray}
{\rm \cite[Lemma 3.4]{FeFa}};
\begin{eqnarray}\label{eq2.20}
\left \langle\frac{\partial^4V}{\partial x^2\partial y^2},-\Delta V\right \rangle \; \geq \; \left\| \frac{\partial^3V}{\partial x^2\partial y}\right \|^2+\left\| \frac{\partial^3V}{\partial x\partial y^2}\right \|^2,
\end{eqnarray}
{\rm \cite[Eq. (2.31)]{fairweather2010}}; and
\begin{eqnarray}
\label{eq2.7}
\|V\|_{H^2}\leq C \|\Delta V\|_D,
\end{eqnarray}
{\rm \cite[Eq. (3.20)]{Bialecki4}}.
\end{lemma}

\section{The ADI Crank-Nicolson OSC scheme}
\label{Fulldis}
\setcounter{equation}{0}
Let $\{t_n\}_{n=0}^M$ be a uniform partition of $[0,T]$ such that $t_n=n\Delta t$,
$\Delta t=T/M$, where $M$ is a positive integer and $\Delta t$ is the time step size. We set $t_{n-1/2}=(n -1/2)\Delta t$, $1\leq n\leq M$. Next, we introduce the following notation:
\[V^n(\cdot,\cdot)=V(\cdot,\cdot,t_n),\ \ 0\leq n\leq M, \]
\[\delta_t V^{n}=\frac{V^n-V^{n-1}}{\Delta t},\qquad V^{n-\frac{1}{2}}=\frac{1}{2}(V^n+V^{n-1}),\ \ 1\leq n\leq M. \]
Then, the time fractional derivative $_0^CD_t^{\alpha}u(x,y,t)$ at $t_{n-\frac{1}{2}}$ can be written  
\bea
\label{eq:3-0}
_0^CD_t^{\alpha}u(x,y,t_{n-\frac{1}{2}})
&=&\frac{1}{\Gamma(2-\alpha)}\int_0^{t_{n-\frac{1}{2}}}\frac{\partial^2 u(x,y,s)}{\partial s^2}\frac{ds}{(t_{n-\frac{1}{2}}-s)^{\alpha-1}}\\
\nonumber\\
&=&{\cal J}^{n-\frac{1}{2}}_{\alpha}(u)+R_{\alpha}^{n-\frac{1}{2}},\quad 1\leq n\leq M,\nonumber
\eea
where
\beq
\label{eq:3-1}
{\cal J}^{n-\frac{1}{2}}_{\alpha}(u)=\frac{\Delta t^{1-\alpha}}{\Gamma (3-\alpha)}\left[b_0\delta_t u^{n}-\sum\limits_{j=1}^{n-1}(b_{n-j-1}-b_{n-j})\delta_t u^{j}-b_{n-1}\phi\right],
\eeq
with
$
b_j=(j+1)^{2-\alpha}-j^{2-\alpha},\; j\geq 0,
$
$
\phi(x,y) =D_t u(x,y,0)
$
from (\ref{eq:2}).  The quantity ${\cal J}^{n-\frac{1}{2}}_{\alpha}(u)$ is the L1-approximation of the Caputo derivative at $t_{n-\frac{1}{2}}$, with
truncation error, $R_{\alpha}^{n-\frac{1}{2}}$, satisfying
\begin{equation}
\label{eq:3-2}
\left |R_{\alpha}^{n-\frac{1}{2}}\right|\leq C{\Delta t}^{3-\alpha},\quad 1\leq n\leq M ;
\end{equation}
see \cite{GaSu,sunzhizhong2012}.
The coefficients $b_j$ possess the following properties
which are required in subsequent analyses.
\begin{lemma}\label{lem:1}{\rm \cite{sunzhizhong2012}}
The coefficients $b_j$, $j\ge 0$, satisfy:
\begin{eqnarray}\label{eq3-2}
&&(i) \ 1=b_0>b_1>\cdots>b_n>b_{n+1}>\cdots\rightarrow 0;\nonumber \\
&&(ii) \ (2-\alpha)(j+1)^{1-\alpha}< b_j<(2-\alpha)j^{1-\alpha},\quad j\geq 1;\nonumber \\
&&(iii) \ \sum\limits_{j=0}^n(b_j-b_{j+1})+b_{n+1}=1;\nonumber \\
&&(iv) \  \sum\limits_{j=1}^{n}b_{j-1}=n^{2-\alpha}.\nonumber
\end{eqnarray}
\end{lemma}

With the approximation of the Caputo derivative given by (\ref{eq:3-1}), the Crank-Nicolson OSC scheme for the approximation of (\ref{eq:1}) consists in find $U_h^n\in\mathcal{M}(\delta)$, $ n=1, 2, \cdots, M$, such that, for $1\leq n\leq M,$
\begin{equation}
\label{eq:3-3}
\qquad \frac{\Delta t^{1-\alpha}}{\Gamma (3-\alpha)}[b_0\delta_t U_h^{n}
-\sum\limits_{j=1}^{n-1}(b_{n-j-1}-b_{n-j})\delta_t U_h^{j}-b_{n-1}\phi]
=\Delta U_h^{n-\frac{1}{2}}+f^{{n-\frac{1}{2}}} \quad {\rm on}\ \Lambda,
\end{equation}
where $f^{n-\frac{1}{2}}=f(\cdot,\cdot,t_{n-\frac{1}{2}})$; $U_h^0$ is prescribed later.

With $E_h^n=U_h^n-U_h^{n-1}$, we write (\ref{eq:3-3}) in the form
\begin{eqnarray}
\label{eq:3-4}
\lefteqn{\frac{\Delta t^{-\alpha}}{\Gamma (3-\alpha)}[E_h^n-\sum\limits_{j=1}^{n-1}(b_{n-j-1}-b_{n-j})E_h^j-\Delta t b_{n-1}\phi]}
\\
&=&\frac{1}{2}\Delta E_h^{n}+\Delta U_h^{n-1}+f^{{n-\frac{1}{2}}} \quad {\rm on}\ \Lambda,\quad  1\leq n\leq M,
\nonumber
\end{eqnarray}
since $b_0 = 1$ from Lemma \ref{lem:1}(i).
Let
\begin{equation}
\label{eqmu}
\mu=\Gamma (3-\alpha) \Delta t^{\alpha}.
\end{equation}
On multiplying (\ref{eq:3-4}) by $\mu$ and rearranging terms, we obtain
\begin{eqnarray}
\label{eq:3-5}
E_h^n-\frac{\mu}{2}\Delta E^n_h
=\sum\limits_{j=1}^{n-1}(b_{n-j-1}-b_{n-j})E_h^j+\Delta t b_{n-1}\phi+\mu\Delta U_h^{n-1}+\mu f^{{n-\frac{1}{2}}}&&
\\
{\rm on}\ \Lambda,\quad  1\leq n\leq M.&&
\nonumber
\end{eqnarray}
On adding the term
\[
\frac{\mu^2}{4}\frac{\partial^4E^n_h}{\partial x^2\partial y^2}
\]
to the left-hand side of (\ref{eq:3-5}), we obtain:
\begin{eqnarray}
\label{eq:3-6}
\left[1-\frac{\mu}{2}\Delta +\frac{\mu^2}{4}\frac{\partial^4}{\partial x^2\partial y^2}\right ]E^n_h=F^n \quad {\rm on}\ \Lambda,\quad  1\leq n\leq M,
\end{eqnarray}
where
\[
F^n=\sum\limits_{j=1}^{n-1}(b_{n-j-1}-b_{n-j})E_h^j+\Delta t b_{n-1}\phi+\mu\Delta U_h^{n-1}+\mu f^{{n-\frac{1}{2}}}+\frac{\mu^2}{4}\frac{\partial^4E^n_h}{\partial x^2\partial y^2},
\]
the basis of the  ADI OSC Crank-Nicolson method for approximating (\ref{eq:1}).

To write (\ref{eq:3-6}) as an ADI method in matrix-vector form, let $\{\chi_i\}_{i=1}^{M_x}$ and $\{\psi_j\}_{j=1}^{M_y}$ be bases for the subspaces $\mathcal{M}(r,\delta_x)$ and $\mathcal{M}(r,\delta_y)$, respectively, where $M_x=(r-1)N_x$ and $M_y=(r-1)N_y$, and set
\[
U^n_h(x,y)=\sum\limits_{i=1}^{M_x}\sum\limits_{j=1}^{M_y}\gamma_{ij}^{(n)}\chi_i(x)\psi_j(y).
\]
We let
\[
{\bf \Upsilon}^{(n)}=
\left[\gamma^{(n)}_{11},\gamma^{(n)}_{12},\cdots,\gamma^{(n)}_{1M_y},
\gamma^{(n)}_{21},\gamma^{(n)}_{22},\cdots,\gamma^{(n)}_{2M_y},
\gamma^{(n)}_{31},\cdots,\gamma^{(n)}_{M_xM_y}\right]^T,
\]
\[
{\bf \mathbf{F}}^{(n)}=[F^n(\xi_1^x,\xi_1^y),F^n(\xi_1^x,\xi_2^y),\cdots,
F^n(\xi_1^x,\xi_{M_y}^y),F^n(\xi_2^x,\xi_1^y),\cdots,F^n(\xi_{M_x}^x,\xi_{M_y}^y)]^T
\]
and define the matrices
\[
\begin{array}{ll}
A_x=\left[-\chi''_j(\xi_i^x)\right]_{i,j=1}^{M_x},& A_y=\left[-\psi''_j(\xi_i^y)\right]_{i,j=1}^{M_y},\\\\
B_x=\left[\chi_j(\xi_i^x)\right]_{i,j=1}^{M_x},&
B_y=\left[\psi_j(\xi_i^y)\right]_{i,j=1}^{M_y}.
\end{array}
\]
Then the algebraic problem comprises determining
${{\mbox{\boldmath$\nu$}}}^{(n)}={\mbox{\boldmath$\Upsilon$}}^{(n)}-{\mbox{\boldmath$\Upsilon$}}^{(n-1)}$ from
\begin{eqnarray}\label{eq:3-11}
\left[\left(B_x+\frac{\mu}{2}A_x\right)\otimes I_{M_y}\right]{\widehat{{\mbox{\boldmath$\nu$}}}}^{(n)}
={\bf \mathbf{F}}^{(n)},
\end{eqnarray}
and
\begin{eqnarray}\label{eq:3-12}
\left[I_{M_x}\otimes \left(B_y+\frac{\mu}{2}A_y\right)\right]
{{\mbox{\boldmath$\nu$}}}^{(n)}={\widehat{{\mbox{\boldmath$\nu$}}}}^{(n)},
\end{eqnarray}
where $\otimes$ denotes the matrix tensor product, and ${\widehat{\mbox{\boldmath$\nu$}}}^{(n)}$ is an auxiliary vector, cf., \cite{fairweather2010}.
Thus, it follows on using properties of $\otimes$ that  ${\mbox{\boldmath$\nu$}}^{(n)}$ is determined by solving the two sets of independent one-dimensional problems, (\ref{eq:3-11}) and (\ref{eq:3-12}). With standard choices of bases for the spaces
$\mathcal{M}(r,\delta_x)$ and $\mathcal{M}(r,\delta_y)$, these linear systems
have an almost block diagonal
structure, and can be solved efficiently using algorithms described in
\cite{FaGl}, for example. Clearly, the computation of ${\mbox{\boldmath$\nu$}}^{(n)}$ is highly parallel.

\section{Stability analysis}
\setcounter{equation}{0}
In this section, we derive stability results in the $H^l$-norm, $l=0, 1, 2$.
\subsection{The $H^1$ stability analysis}
An $H^1$ stability result for (\ref{eq:3-6}) is proved in the following theorem.
\begin{theorem}
\label{th:stability}
The ADI OSC Crank-Nicolson method (\ref{eq:3-6}) is stable with respect to the $H^1$ norm. Specifically, for $U^n_h\in \mathcal{M}(\delta)$, $1\leq n\leq M$,
\begin{equation}
\label{eq:stability}
\left\|\nabla U_h^n\right\|^2\leq C \left\|\nabla U_h^0\right\|^2+\frac{t_n^{2-\alpha}}{\Gamma(3-\alpha)}\left\|\phi\right\|_D^2+\Delta t \sum\limits_{j=1}^{n}\Gamma(2-\alpha)t_n^{\alpha-1}\left\|f^{j-\frac{1}{2}}\right\|_D^2.
\end{equation}
\end{theorem}

\begin{proof}  First note that (\ref{eq:3-6}) can be written as
\begin{eqnarray}
&&\delta_t U_h^{n}-\frac{\mu}{\Delta t}\Delta U^{n-\frac{1}{2}}_h
+\frac{\mu^2}{4}\frac{\partial^4\delta_t U_h^{n}}{\partial x^2\partial y^2}
\label{eq4.2}
\\
&=&\sum\limits_{j=1}^{n-1}(b_{n-j-1}-b_{n-j})\delta_t U_h^{j}+ b_{n-1}\phi+\frac{\mu}{\Delta t} f^{{n-\frac{1}{2}}} \quad {\rm on}\ \Lambda,\quad  1\leq n\leq M,
\nonumber
\end{eqnarray}
or, on substituting (\ref{eqmu}) into (\ref{eq4.2}) and rearranging terms,
\begin{eqnarray}
\label{eq4.3}
\frac{\Delta t^{1-\alpha}}{\Gamma (3-\alpha)}\delta_t U_h^{n}-\Delta U^{n-\frac{1}{2}}_h
+\frac{\Gamma(3-\alpha)\Delta t^{1+\alpha}}{4}\frac{\partial^4\delta_t U_h^{n}}{\partial x^2\partial y^2}\qquad\qquad&&
\\
=\frac{\Delta t^{1-\alpha}}{\Gamma (3-\alpha)}\sum\limits_{j=1}^{n-1}(b_{n-j-1}-b_{n-j})\delta_t U_h^{j}+ \frac{\Delta t^{1-\alpha}}{\Gamma (3-\alpha)}b_{n-1}\phi+f^{{n-\frac{1}{2}}}&&
\nonumber \\
\quad {\rm on}\ \Lambda,\quad  1\leq n\leq M.&&
\nonumber
\end{eqnarray}
Taking the discrete inner product of (\ref{eq4.3}) with $\delta_t U_h^{n}$ yields
\begin{eqnarray}
\label{eq4.4}
&&\frac{\Delta t^{1-\alpha}}{\Gamma (3-\alpha)}\left\langle\delta_t U_h^{n},\delta_t U_h^{n}\right\rangle-\left\langle\Delta U^{n-\frac{1}{2}}_h,\delta_t U_h^{n}\right\rangle
\\ && \quad +\frac{\Gamma(3-\alpha)\Delta t^{1+\alpha}}{4}\left\langle\frac{\partial^4\delta_t U_h^{n}}{\partial x^2\partial y^2},\delta_t U_h^{n}\right\rangle\nonumber
\\&=&\frac{\Delta t^{1-\alpha}}{\Gamma (3-\alpha)}\sum\limits_{j=1}^{n-1}(b_{n-j-1}-b_{n-j})\left\langle\delta_t U_h^{j},\delta_t U_h^{n}\right\rangle\nonumber \\ && \quad + \frac{\Delta t^{1-\alpha}}{\Gamma (3-\alpha)}b_{n-1}\left\langle\phi,\delta_t U_h^{n}\right\rangle+\left\langle f^{{n-\frac{1}{2}}},\delta_t U_h^{n}\right\rangle, \quad 1\leq n\leq M.
\nonumber
\end{eqnarray}
The first term on the left-hand side of (\ref{eq4.4}) can be written as
\begin{eqnarray}\label{eq4.5}
\left\langle\delta_t U_h^{n},\delta_t U_h^{n}\right\rangle \;\;=\;\; \left\|\delta_t U_h^{n} \right\|^2_D.
\end{eqnarray}
A straightforward calculation shows that the second term on the left-hand
side of (\ref{eq4.4}) gives
\begin{eqnarray}\label{eq4.7}
-\left \langle \Delta U_h^{n-\frac{1}{2}},\delta_t U_h^{n}\right \rangle \; = \; \frac{1}{2} \delta_t\left \langle -\Delta U_h^{n}, U_h^{n}\right \rangle.
\end{eqnarray}
and, from Lemma \ref{lem:2.1}, we have, for the third term,
\begin{eqnarray}\label{eq4.6}
\left \langle\frac{\partial^4\left[\delta_t U_h^{n}\right]}{\partial x^2\partial y^2},\delta_t U_h^{n}\right \rangle \; \geq \; \left\| \frac{\partial^2\delta_t U_h^{n}}{\partial x\partial y}\right \|^2 \;\;
\geq \;\; 0.
\end{eqnarray}

On substituting (\ref{eq4.5}) and (\ref{eq4.7}) into (\ref{eq4.4}) and dropping the non-negative term, $\left\| \ds\frac{\partial^2\delta_t U_h^{n}}{\partial x\partial y}\right \|^2$,  we obtain
\begin{eqnarray}
\label{eq4.8}
&&\frac{\Delta t^{1-\alpha}}{\Gamma (3-\alpha)}\left\|\delta_t U_h^{n} \right\|^2_D+\frac{1}{2} \delta_t\left \langle -\Delta U_h^{n}, U_h^{n}\right \rangle
\\&\leq&\frac{\Delta t^{1-\alpha}}{\Gamma (3-\alpha)}\sum\limits_{j=1}^{n-1}(b_{n-j-1}-b_{n-j})\left\langle\delta_t U_h^{j},\delta_t U_h^{n}\right\rangle\nonumber \\ && \quad + \frac{\Delta t^{1-\alpha}}{\Gamma (3-\alpha)}b_{n-1}\left\langle\phi,\delta_t U_h^{n}\right\rangle+\left\langle f^{{n-\frac{1}{2}}},\delta_t U_h^{n}\right\rangle, \quad 1\leq n\leq M.
\nonumber
\end{eqnarray}
Multiplying (\ref{eq4.8}) by $2\Delta t$, and using the fact that, from Lemma \ref{lem:1}(i),  $b_{n-1}>0$ and $b_{n-j-1}-b_{n-j}>0$,  we obtain, for $ 1\leq n\leq M$,
\begin{eqnarray}
\label{eq4.9}
&&\frac{2\Delta t^{2-\alpha}}{\Gamma (3-\alpha)}\left\|\delta_t U_h^{n} \right\|^2_D
+\left \langle -\Delta U_h^{n}, U_h^{n}\right \rangle
\\
&\leq&\left \langle -\Delta U_h^{n-1}, U_h^{n-1}\right \rangle+\frac{2\Delta t^{2-\alpha}}{\Gamma (3-\alpha)}\sum\limits_{j=1}^{n-1}(b_{n-j-1}-b_{n-j})\left\langle\delta_t U_h^{j},\delta_t U_h^{n}\right\rangle
\nonumber\\ && \quad + \frac{2\Delta t^{2-\alpha}}{\Gamma (3-\alpha)}b_{n-1}\left\langle\phi,\delta_t U_h^{n}\right\rangle+2\Delta t\left\langle f^{{n-\frac{1}{2}}},\delta_t U_h^{n}\right\rangle.
\nonumber
\end{eqnarray}
On using the Cauchy-Schwarz inequality and the triangle inequality, (\ref{eq4.9}) can be rewritten as
\begin{eqnarray}
\label{eq4.10}
&&\frac{2\Delta t^{2-\alpha}}{\Gamma (3-\alpha)}\left\|\delta_t U_h^{n} \right\|^2_D
+\left \langle -\Delta U_h^{n}, U_h^{n}\right \rangle
\\
&\leq&\left \langle -\Delta U_h^{n-1}, U_h^{n-1}\right \rangle+\frac{\Delta t^{2-\alpha}}{\Gamma (3-\alpha)}\sum\limits_{j=1}^{n-1}(b_{n-j-1}-b_{n-j})\left[\left\|\delta_t U_h^{j}\right\|_D^{2}+\left\|\delta_t U_h^{n}\right\|_D^2\right]\nonumber
\\
&& \quad + \frac{\Delta t^{2-\alpha}}{\Gamma (3-\alpha)}b_{n-1}\left[\left\|\phi\right\|_D^{2}+\left\|\delta_t U_h^{n}\right\|_D^2\right]+2\Delta t\left\langle f^{{n-\frac{1}{2}}},\delta_t U_h^{n}\right\rangle
\nonumber
\\
&=&\left \langle -\Delta U_h^{n-1}, U_h^{n-1}\right \rangle+\frac{\Delta t^{2-\alpha}}{\Gamma (3-\alpha)}\sum\limits_{j=1}^{n-1}(b_{n-j-1}-b_{n-j})\left\|\delta_t U_h^{j}\right\|_D^{2}\nonumber \\ && \quad+\frac{\Delta t^{2-\alpha}}{\Gamma (3-\alpha)}\left[\sum\limits_{j=1}^{n-1}(b_{n-j-1}-b_{n-j})+b_{n-1}\right ]\left\|\delta_t U_h^{n}\right\|_D^2
\nonumber \\ && \quad + \frac{\Delta t^{2-\alpha}}{\Gamma (3-\alpha)}b_{n-1}\left\|\phi\right\|_D^{2}+2\Delta t\left\langle f^{{n-\frac{1}{2}}},\delta_t U_h^{n}\right\rangle
, \quad 1\leq n\leq M.
\nonumber
\end{eqnarray}
Note that, from Lemma \ref{lem:1}(i) and (iii),
\[
\sum\limits_{j=1}^{n-1}(b_{n-j-1}-b_{n-j})+b_{n-1}=b_0=1,
\]
 so that (\ref{eq4.10}) becomes
\begin{eqnarray}
\label{eq4.11}
&&\frac{2\Delta t^{2-\alpha}}{\Gamma (3-\alpha)}\left\|
\delta_t U_h^{n} \right\|^2_D+\left \langle -\Delta U_h^{n},
U_h^{n}\right \rangle
\\
&\leq&\left \langle -\Delta U_h^{n-1}, U_h^{n-1}\right \rangle+\frac{\Delta t^{2-\alpha}}{\Gamma (3-\alpha)}\sum\limits_{j=1}^{n-1}(b_{n-j-1}-b_{n-j})\left\|\delta_t U_h^{j}\right\|_D^{2}\nonumber
\\
&& \quad+\frac{\Delta t^{2-\alpha}}{\Gamma (3-\alpha)}\left\|\delta_t U_h^{n}\right\|_D^2
+ \frac{\Delta t^{2-\alpha}}{\Gamma (3-\alpha)}b_{n-1}\left\|\phi\right\|_D^{2}\nonumber \\ && \quad +2\Delta t\left\langle f^{{n-\frac{1}{2}}},\delta_t U_h^{n}\right\rangle
, \quad 1\leq n\leq M.
\nonumber
\end{eqnarray}
On reformulating (\ref{eq4.11}), we obtain
\begin{eqnarray}
\label{eq4.12}
&&\left \langle
-\Delta U_h^{n}, U_h^{n}\right \rangle+\displaystyle{\frac{\Delta t^{2-\alpha}}{\Gamma (3-\alpha)}}\sum\limits_{j=1}^{n}b_{n-j}
\left\|\delta_t U_h^{j} \right\|^2_D
\\
&\leq&\left \langle -\Delta U_h^{n-1}, U_h^{n-1}\right \rangle+\frac{\Delta t^{2-\alpha}}{\Gamma (3-\alpha)}\sum\limits_{j=1}^{n-1}b_{n-j-1}\left\|\delta_t U_h^{j}\right\|_D^{2}
 \nonumber \\ && \quad
+ \frac{\Delta t^{2-\alpha}}{\Gamma (3-\alpha)}b_{n-1}\left\|\phi\right\|_D^{2}+2\Delta t\left\langle f^{{n-\frac{1}{2}}},\delta_t U_h^{n}\right\rangle
, \quad 1\leq n\leq M.
\nonumber
\end{eqnarray}
For convenience, we define $G^n$ by:
\begin{equation}
\label{eq4.13}
\left \{
\begin{array}{l}
G^0=\left \langle-\Delta U_h^{0}, U_h^{0}\right \rangle,
\\
G^n=\left \langle-\Delta U_h^{n}, U_h^{n}\right \rangle+\displaystyle{\frac{\Delta t^{2-\alpha}}{\Gamma (3-\alpha)}}\sum\limits_{j=1}^{n}b_{n-j}
\left\|\delta_t U_h^{j} \right\|^2_D, \quad n \ge 1.\\

\end{array}
\right .
\end{equation}
Then (\ref{eq4.12}) is equivalent to
\begin{eqnarray}
\label{eq4.14}
G^n
&\leq& G^{n-1}
+ \frac{\Delta t^{2-\alpha}}{\Gamma (3-\alpha)}b_{n-1}\left\|\phi\right\|_D^{2}+2\Delta t\left\langle f^{{n-\frac{1}{2}}},\delta_t U_h^{n}\right\rangle
 \\
&\leq& G^{n-2}+ \frac{\Delta t^{2-\alpha}}{\Gamma (3-\alpha)}b_{n-2}\left\|\phi\right\|_D^{2}+2\Delta t\left\langle f^{{n-1-\frac{1}{2}}},\delta_t U_h^{n-1}\right\rangle
\nonumber \\
&& \quad + \frac{\Delta t^{2-\alpha}}{\Gamma (3-\alpha)}b_{n-1}\left\|\phi\right\|_D^{2}
+2\Delta t\left\langle f^{{n-\frac{1}{2}}},\delta_t U_h^{n}\right\rangle
\nonumber \\
&&\hspace{1.in}\ldots \ldots
\nonumber\\\nonumber \\
&\leq& G^{0}+ \frac{\Delta t^{2-\alpha}}{\Gamma (3-\alpha)}\left[\sum\limits_{j=1}^{n}b_{n-j}\right ]\left\|\phi\right\|_D^{2}+2\Delta t\sum\limits_{j=1}^{n}\left\langle f^{{j-\frac{1}{2}}},\delta_t U_h^{j}\right\rangle.
\nonumber
\\
&\leq& G^{0}+ \frac{ t_n^{2-\alpha}}{\Gamma (3-\alpha)}\left\|\phi\right\|_D^{2}+2\Delta t\sum\limits_{j=1}^{n}\left\langle f^{{j-\frac{1}{2}}},\delta_t U_h^{j}\right\rangle,
\nonumber
\end{eqnarray}
using Lemma \ref{lem:1}(iv) in the last step.
On using the Cauchy-Schwarz inequality and the Young's inequality (\ref{eqyoung}),
the last term on the right hand side of (\ref{eq4.14}) may be bounded as
\begin{eqnarray}
\label{eq4.16}
\lefteqn{2\Delta t\sum\limits_{j=1}^{n}\left\langle f^{{j-\frac{1}{2}}},\delta_t U_h^{j}\right\rangle}
\\
&\leq& \Delta t\sum\limits_{j=1}^{n}\left[\frac{\Gamma (3-\alpha)}{\Delta t^{1-\alpha}b_{n-j}}\left\| f^{j-\frac{1}{2}}\right\|^2_D+\frac{\Delta t^{1-\alpha}b_{n-j}}{\Gamma (3-\alpha)}\left\|\delta_t U_h^{j}\right\|_D^2\right].
\nonumber
\end{eqnarray}
On substituting (\ref{eq4.13}) and (\ref{eq4.16}) into (\ref{eq4.14}) and simplifying the resulting expression,
we obtain, for $1 \le n \le M$,
\begin{eqnarray}
\label{eq4.18}
\left \langle
-\Delta U_h^{n}, U_h^{n}\right \rangle
&\leq&\left \langle -\Delta U_h^{0}, U_h^{0}\right \rangle
+ \frac{t_n^{2-\alpha}}{\Gamma (3-\alpha)}\left\|\phi\right\|_D^{2}
\\
&&\quad +
\Delta t\sum\limits_{j=1}^{n}\frac{\Delta t^{\alpha-1}\Gamma (3-\alpha)}{b_{n-j}}\left\| f^{j-\frac{1}{2}}\right\|^2_D
, \quad 1\leq n\leq M.
\nonumber
\end{eqnarray}
Also, since, from Lemma \ref{lem:1}(ii),
\[
b_{n-j}\geq(2-\alpha)(n-j+1)^{1-\alpha}\geq(2-\alpha)n^{1-\alpha},
\]
the last term on the right hand side of (\ref{eq4.18}) can be bounded as
\begin{eqnarray}
\label{eq4.20}
\Delta t\sum\limits_{j=1}^{n}\frac{\Delta t^{\alpha-1}\Gamma (3-\alpha)}{b_{n-j}}\left\| f^{j-\frac{1}{2}}\right\|^2_D
&\leq& \Delta t\sum\limits_{j=1}^{n}\frac{\Delta t^{\alpha-1}\Gamma (3-\alpha)}{(2-\alpha)n^{1-\alpha}}\left\| f^{j-\frac{1}{2}}\right\|^2_D
\\
&\leq& \Delta t \sum\limits_{j=1}^{n}\Gamma (2-\alpha)t_n^{\alpha-1}\left\| f^{j-\frac{1}{2}}\right\|^2_D.
\nonumber
\end{eqnarray}
Using (\ref{eq2.4}) and (\ref{eq2.5}), it can be shown that there exist positive constants $C_1$ and $C_2$ such that
\begin{eqnarray}\label{eq4.21}
\left \langle-\Delta U_h^n,U_h^n\right \rangle \; \geq \; C_1\left\|\nabla U_h^n\right\|^2, \quad \left \langle-\Delta U_h^0,U_h^0\right \rangle \; \leq \; C_2\left\| \nabla U_h^0\right\|^2.
\end{eqnarray}
Substituting (\ref{eq4.20}) and (\ref{eq4.21}) into (\ref{eq4.18}), and rearranging, we obtain
\begin{equation}
\label{eq:4.22}
\left\|\nabla U_h^n\right\|^2\leq C \left\|\nabla U_h^0\right\|^2+\frac{t_n^{2-\alpha}}{\Gamma(3-\alpha)}\left\|\phi\right\|_D^2+\Delta t \sum\limits_{j=1}^{n}\Gamma(2-\alpha)t_n^{\alpha-1}\left\|f^{j-\frac{1}{2}}\right\|_D^2,
\end{equation}
which completes the proof. \qquad\end{proof}

\subsection{The $H^2$ stability analysis}
An $H^2$ stability estimate is derived in the following theorem.
\begin{theorem}
\label{th:stability2}
The ADI OSC Crank-Nicolson method (\ref{eq:3-6}) is stable with respect to the $H^2$ norm. More precisely, for $U^n_h\in \mathcal{M}(\delta)$, $1\leq n, q\leq M$, we have
\begin{eqnarray}
\label{eq:stability2}
\left\|U_h^n\right\|_{H^2}^2\leq C \left[\left\|\Delta U_h^0\right\|_D^2+t_q^{2-2\alpha}\left\|\phi\right\|_D^2+\left\| f^{\frac{1}{2}}\right\|_D^2+\left\|f^{n-\frac{1}{2}}\right\|_D^2+\Delta t \sum\limits_{j=2}^{n}\left\|\delta_t f^{j-\frac{1}{2}}\right\|_D^2\right].
\nonumber
\end{eqnarray}
\end{theorem}
\begin{proof}  Taking the inner product of (\ref{eq4.3}) with $-\Delta\delta_t U_h^{n}$, we obtain
\begin{eqnarray}
\label{eq4.23}
&&\frac{\Delta t^{1-\alpha}}{\Gamma (3-\alpha)}\left\langle\delta_t U_h^{n},-\Delta\delta_t U_h^{n}\right\rangle-\left\langle\Delta U_h^{n-\frac{1}{2}},-\Delta\delta_t U_h^{n}\right\rangle
\\ && \quad +\frac{\Gamma(3-\alpha)\Delta t^{1+\alpha}}{4}\left\langle\frac{\partial^4\delta_t U_h^{n}}{\partial x^2\partial y^2},-\Delta\delta_t U_h^{n}\right\rangle\nonumber
\\&=&\frac{\Delta t^{1-\alpha}}{\Gamma (3-\alpha)}\sum\limits_{j=1}^{n-1}(b_{n-j-1}-b_{n-j})\left\langle\delta_t U_h^{j},-\Delta\delta_t U_h^{n}\right\rangle
\nonumber\\ &&
+ \frac{\Delta t^{1-\alpha}}{\Gamma (3-\alpha)}b_{n-1}\left\langle\phi,-\Delta\delta_t U_h^{n}\right\rangle+\left\langle f^{{n-\frac{1}{2}}},-\Delta\delta_t U_h^{n}\right\rangle, \quad 1\leq n\leq M.
\nonumber
\end{eqnarray}
On using (\ref{eq2.4}) with $U=\delta_t U_h^{n}$, we have, in the first term on the right hand side,
\begin{eqnarray}
\label{eq4.24}
\left\langle\delta_t U_h^{n},-\Delta\delta_t U_h^{n}\right\rangle \;\;\geq\;\; C\left\|\nabla\delta_t U_h^{n} \right\|^2.
\end{eqnarray}
The second term on the left-hand side of (\ref{eq4.23}) can be written as
\begin{eqnarray}
\label{eq4.25}
\left \langle \Delta U_h^{n-\frac{1}{2}},\Delta\delta_t U_h^{n}\right \rangle \; = \; \frac{1}{2} \delta_t\left \langle \Delta U_h^{n}, \Delta U_h^{n}\right \rangle\; = \;\frac{1}{2}\delta_t\left\|\Delta U_h^{n} \right\|^2_D,
\end{eqnarray}
>From (\ref{eq2.20}) with $V=\delta_t U_h^{n}$,
\begin{equation}\label{eq4.26}
\left \langle\frac{\partial^4\delta_t U_h^{n}}{\partial x^2\partial y^2},-\Delta\delta_t U_h^{n}\right \rangle
 \geq \left\| \frac{\partial^3\delta_t U_h^{n}}{\partial x^2\partial y}\right \|^2+ \left\| \frac{\partial^3\delta_t U_h^{n}}{\partial x\partial y^2}\right \|^2 \;\;
\geq \;\; 0.
\end{equation}
On substituting (\ref{eq4.24})--(\ref{eq4.26}) into (\ref{eq4.23}), dropping the non-negative terms
$\left\| \ds\frac{\partial^3\delta_t U_h^{n}}{\partial x^2\partial y}\right \|^2$ and $\left\| \ds\frac{\partial^3\delta_t U_h^{n}}{\partial x\partial y^2}\right \|^2$, we obtain
\begin{eqnarray}
\label{eq4.27}
&&\frac{C_1\Delta t^{1-\alpha}}{\Gamma (3-\alpha)}\left\|\nabla\delta_t U_h^{n} \right\|^2+
\frac{1}{2}\delta_t\left\|\Delta U_h^{n} \right\|^2_D
\\&\leq&\frac{\Delta t^{1-\alpha}}{\Gamma (3-\alpha)}\sum\limits_{j=1}^{n-1}(b_{n-j-1}-b_{n-j})\left\langle\delta_t U_h^{j},-\Delta\delta_t U_h^{n}\right\rangle
\nonumber \\ &&
+ \frac{\Delta t^{1-\alpha}}{\Gamma (3-\alpha)}b_{n-1}\left\langle\phi,-\Delta\delta_t U_h^{n}\right\rangle
+\left\langle f^{{n-\frac{1}{2}}},-\Delta\delta_t U_h^{n}\right\rangle, \quad 1\leq n\leq M.
\nonumber
\end{eqnarray}
Since $b_{n-j-1}-b_{n-j}>0$ from Lemma \ref{lem:1}$(i)$,
\bea
\label{eq:4.27a}
\lefteqn{\sum\limits_{j=1}^{n-1}(b_{n-j-1}-b_{n-j})\left\langle\delta_t U_h^{j},-\Delta\delta_t U_h^{n}\right\rangle }\\
&\leq& \sum\limits_{j=1}^{n-1}(b_{n-j-1}-b_{n-j})\left|\left\langle\delta_t U_h^{j},-\Delta\delta_t U_h^{n}\right\rangle\right|\nonumber \\
&\leq&
\sum\limits_{j=1}^{n-1}(b_{n-j-1}-b_{n-j})\left[\frac{1}{\varepsilon}\left\|\nabla\delta_t U_h^{j}\right\|^2+\varepsilon\left\|\nabla\delta_t U_h^{n}\right\|^2\right]\nonumber
\\
&\leq& \frac{1}{\varepsilon}\sum\limits_{j=1}^{n-1}(b_{n-j-1}-b_{n-j})\left\|\nabla\delta_t U_h^{j}\right\|^2 + \varepsilon\left\|\nabla\delta_t U_h^{n}\right\|^2,\nonumber
\eea
on using (\ref{eq2.5}), (\ref{eqyoung}) and the fact that $\sum\limits_{j=1}^{n-1}(b_{n-j-1}-b_{n-j})=1-b_{n-1}<1$ from Lemma \ref{lem:1}$(iii)$.
We multiply(\ref{eq4.27}) by $2\Delta t$, use (\ref{eq:4.27a}) and rearrange terms to obtain
\begin{eqnarray}
\label{eq4.29}
&&\left\|\Delta U_h^{n} \right\|^2_D+
\frac{C_3\Delta t^{2-\alpha}}{\Gamma (3-\alpha)}\sum\limits_{j=1}^{n}b_{n-j}\left\|\nabla\delta_t U_h^{j} \right\|^2
\\&\leq&\left\|\Delta U_h^{n-1} \right\|^2_D+\frac{C_4\Delta t^{2-\alpha}}{\Gamma (3-\alpha)}\sum\limits_{j=1}^{n-1}b_{n-j-1}\left\|\nabla\delta_t U_h^{j}\right\|^2
\nonumber \\ && \quad
+ \frac{2\Delta t^{2-\alpha}}{\Gamma (3-\alpha)}b_{n-1}\left\langle\phi,-\Delta\delta_t U_h^{n}\right\rangle
 +2\Delta t\left\langle f^{{n-\frac{1}{2}}},-\Delta\delta_t U_h^{n}\right\rangle, \quad 1\leq n\leq M.
\nonumber
\end{eqnarray}
Therefore, using arguments similar to those in (\ref{eq4.14}),  we have, for $1\leq n\leq M$,
\begin{eqnarray}
\label{eq4.30}
&&\left\|\Delta U_h^{n} \right\|^2_D+
\frac{C_3\Delta t^{2-\alpha}}{\Gamma (3-\alpha)}\sum\limits_{j=1}^{n}b_{n-j}\left\|\nabla\delta_t U_h^{j} \right\|^2
\\&\leq&
\left\|\Delta U_h^{0} \right\|^2_D
+ \frac{2\Delta t^{2-\alpha}}{\Gamma (3-\alpha)}\sum\limits_{j=1}^{n}b_{j-1}\left\langle\phi,-\Delta\delta_t U_h^{j}\right\rangle
 +2\Delta t\sum\limits_{j=1}^{n}\left\langle f^{{j-\frac{1}{2}}},-\Delta\delta_t U_h^{j}\right\rangle.
\nonumber
\end{eqnarray}
In the Appendix $A$, it is shown that there exist integers $m$ and $q$, with $1\leq m\leq n$, $0\leq q\leq n-1$, such that
\begin{eqnarray}
\label{eq4.32}
\Delta t\sum\limits_{j=1}^{n}b_{j-1}\left\langle\phi,-\Delta\delta_t U_h^{j}\right\rangle &=&\sum\limits_{j=1}^{n}b_{j-1}\left\langle\phi,\Delta U_h^{j}-\Delta U_h^{j-1}\right\rangle
\\
&\leq&
b_{q}\left\langle\phi,\Delta U_h^{m}-\Delta U_h^{0}\right\rangle, \nonumber
\end{eqnarray}
since
$b_0=1$, $b_{q}\leq(2-\alpha)q^{1-\alpha},q>0,$
and $t_1=\Delta t$. Using (\ref{eqyoung}) and (\ref{eq4.32}), the Cauchy-Schwarz and triangle inequality, for $
1\leq q\leq n-1, 1\leq m\leq n,$ we then obtain
\begin{eqnarray}
\label{eq4.33}
\lefteqn{\frac{2\Delta t^{2-\alpha}}{\Gamma (3-\alpha)}\sum\limits_{j=1}^{n}b_{j-1}\left\langle\phi,-\Delta\delta_t U_h^{j}\right\rangle
\leq
\frac{2}{\Gamma (3-\alpha)}\left\langle t_q^{1-\alpha}\phi,\Delta U_h^{m}-\Delta U_h^{0}\right\rangle}
\\
&\leq&
\frac{1}{\Gamma (3-\alpha)} \left[\frac{ t_q^{2-2\alpha}}{\varepsilon}\left\|\phi\right\|_D^2+\varepsilon\left\|\Delta U_h^{m} \right\|_D^2+\varepsilon\left\|\Delta U_h^{0} \right\|_D^2\right].
 \nonumber
\end{eqnarray}
The last term on the right-hand side of (\ref{eq4.30}) is bounded as in \cite[Eq. (2.35)]{fairweather2010} to obtain
\begin{eqnarray}
\label{eq4.35}
2\Delta t\left|\sum\limits_{j=1}^{n}\left\langle f^{j-\frac{1}{2}},-\Delta\delta_t U_h^{j}\right\rangle\right|
 &\leq&\frac{1}{\varepsilon}\left\|f^{{n-\frac{1}{2}}}\right\|_D^2+\varepsilon\left\|\Delta U_h^{n} \right\|_D^2+\frac{1}{\varepsilon}\left\|f^{{\frac{1}{2}}}\right\|_D^2+\varepsilon\left\|\Delta U_h^{0} \right\|_D^2\hspace{-.4in}
  \\
 &&+\Delta t\sum\limits_{j=2}^{n}\left\|\delta_tf^{{j-\frac{1}{2}}}\right\|_D^2+\Delta t\sum\limits_{j=0}^{n-1}\left\|\Delta
 U_h^j\right\|_D^2.
 \nonumber
\end{eqnarray}

On substituting (4.29) and (4.30) into (4.27), dropping the non-negative second term on the left-hand side of (4.27), and simplifying the resulting expression, we obtain,
 for
$1\leq m\leq n,$ $1\leq q\leq n-1,$ $1\leq n\leq M,$
\begin{eqnarray}
\label{eq4.36}
\left\|\Delta U_h^{n} \right\|^2_D
&\leq& C\left[\left\|\Delta U_h^{0} \right\|^2_D+ t_q^{2-2\alpha}\left\|\phi\right\|_D^2+\left\|f^{{\frac{1}{2}}}\right\|_D^2
+\left\|f^{{n-\frac{1}{2}}}\right\|_D^2\right.
\\ &&
\left .+\Delta t\sum\limits_{j=2}^{n}\left\|\delta_t f^{{j-\frac{1}{2}}}\right\|_D^2\right ]+ C\Delta t\sum\limits_{j=1}^{n-1} \left\|\Delta
 U_h^j\right\|_D^2+C\varepsilon\left\|\Delta U_h^m\right\|_D^2.
\nonumber
\end{eqnarray}
Suppose
\[
\max_{0\le \ell \le n}\left\|\Delta U_h^{\ell}\right\|_D = \left\|\Delta U_h^J\right\|_D.
\]
Then
\begin{eqnarray*}
\left\|\Delta U_h^{n} \right\|^2_D
&\leq& C\left[\left\|\Delta U_h^{0} \right\|^2_D+ t_q^{2-2\alpha}\left\|\phi\right\|_D^2+\left\|f^{{\frac{1}{2}}}\right\|_D^2
+\left\|f^{{n-\frac{1}{2}}}\right\|_D^2\right.
\\ &&
\left .+\Delta t\sum\limits_{j=2}^{n}\left\|\delta_t f^{{j-\frac{1}{2}}}\right\|_D^2\right ]+ C\Delta t\sum\limits_{j=1}^{n-1} \left\|\Delta
 U_h^j\right\|_D^2+C\varepsilon\left\|\Delta U_h^J\right\|_D^2.
\end{eqnarray*}
Thus, since $0\le J \le n$,
\begin{eqnarray*}
\left\|\Delta U_h^{J} \right\|^2_D
&
\leq &C\left[\left\|\Delta U_h^{0} \right\|^2_D+ t_q^{2-2\alpha}\left\|\phi\right\|_D^2+\left\|f^{{\frac{1}{2}}}\right\|_D^2
+\left\|f^{{n-\frac{1}{2}}}\right\|_D^2\right.
\\ &&
\left .+\Delta t\sum\limits_{j=2}^{n}\left\|\delta_t f^{{j-\frac{1}{2}}}\right\|_D^2\right ]+ C\Delta t\sum\limits_{j=1}^{n-1} \left\|\Delta
 U_h^j\right\|_D^2+C\varepsilon\left\|\Delta U_h^J\right\|_D^2,
\end{eqnarray*}
from which it follows that, for $\varepsilon$ sufficiently small,
\begin{eqnarray}
\label{eq4.37c}
\left\|\Delta U_h^{n} \right\|^2_D &\leq& \left\|\Delta U_h^J\right\|_D^2
\leq C\left[\left\|\Delta U_h^{0} \right\|^2_D+ t_q^{2-2\alpha}\left\|\phi\right\|_D^2+\left\|f^{{\frac{1}{2}}}\right\|_D^2
+\left\|f^{{n-\frac{1}{2}}}\right\|_D^2\right.
\\ &&
\left .+\Delta t\sum\limits_{j=2}^{n}\left\|\delta_t f^{{j-\frac{1}{2}}}\right\|_D^2\right ]+ C\Delta t\sum\limits_{j=1}^{n-1} \left\|\Delta
 U_h^j\right\|_D^2.
\nonumber
\end{eqnarray}
Thus, on applying the discrete Gronwall lemma,
\[
\left\|\Delta U_h^{n} \right\|^2_D
\leq C\left[\left\|\Delta U_h^{0} \right\|^2_D+ t_q^{2-2\alpha}\left\|\phi\right\|_D^2+\left\|f^{{\frac{1}{2}}}\right\|_D^2
+\left\|f^{{n-\frac{1}{2}}}\right\|_D^2+\Delta t\sum\limits_{j=2}^{n}\left\|\delta_t f^{{j-\frac{1}{2}}}\right\|_D^2\right ].
\]
\end{proof}
\section{Convergence analysis}
\setcounter{equation}{0}
In this section, we give an analysis of the convergence of the ADI OSC method. For this purpose, we introduce the elliptic projection
$W: ~[0,T]\rightarrow \mathcal{M}(\delta)$ defined by
\begin{equation}\label{eq:W}
\Delta\left(u-W\right)=0\ \ \mbox{on}\ \  \Lambda \times [0,T],
\end{equation}
where $u$ is the solution of (\ref{eq:1})--(\ref{eq:3}).
The following two lemmas  provide estimates for $u-W$ and its time derivatives;
see \cite[Eqs. (2.45), (2.46)]{fairweather2010}.
\begin{lemma}\label{lem5}
If $\partial^iu/\partial t^l \in H^{r+3-j},\; i=0, 1, 2,
\;j=0, 1, 2$, and $W$ is defined by (\ref{eq:W}), then
\begin{equation}\label{eq5.2}
\left \|\frac{{\partial}^l(u-W)}{\partial t^l}\right \|_{H^j}\leq
Ch^{r+1-j}\left \|\frac{{\partial}^lu}{\partial t^l}\right  \|_{H^{r+3-j}},\quad
j=0, 1, 2,\ \ l=0, 1, 2.
\end{equation}
\end{lemma}

\begin{lemma}\label{lem6}
If $\partial^i u/\partial t^i\in H^{r+3}$, for $t\in \left [0,T\right ]$, $i=0, 1, 2$, then
\begin{equation}\label{eq5.3}
\left \|\frac{\partial ^{l+i}(u-W)}{\partial x^{l_1}\partial y^{l_2}\partial t^i}\right\|_{D}\leq
Ch^{r+1-l}\left \|\frac{{\partial}^iu}{\partial t^i}\right \|_{H^{r+3}}, \quad 0\leq l=l_1+l_2\leq 4.
\end{equation}
\end{lemma}

Convergence results for the ADI OSC method are given in the following  theorem.
\begin{theorem}
\label{th:convergence}
Suppose $u$ is the solution of
(\ref{eq:1})--(\ref{eq:3}), and $U_h^n$, $n=1,2,\cdots,M$, satisfies (\ref{eq:3-6}) with $U^0_h=W^0$. If $u\in C^{2,0,3}\cap C^{0,2,3} \cap C^{2,3,1} \cap C^{3,2,1}\cap C^{0,0,4}$ and
$
\partial u/\partial t,$ $
\partial^2 u/\partial t^2,$ $\partial^3 u/\partial t^3\in C\left([0,T],H^{r+3}\right),
$
 then 
\begin{equation}\label{eq:convergence}
\left\|u(t_{n})-U_h^{n}\right\|_{H^j}\leq C\left (h^{r+1-j}+\Delta
t^{3-\alpha}\right ), \quad j =0, 1, 2.
\end{equation}
\end{theorem}
\begin{proof}
With $W$ defined in (\ref{eq:W}), we set
\begin{equation}\label{eq5.5}
\eta^{n}=u^{n}-W^{n}, \quad \zeta^{n}=U_h^n-W^{n},\qquad 0 \leq n \leq M,
\end{equation}
so that
\begin{equation}\label{eq5.6}
u^n-U_h^{n}=\eta^{n}-\zeta^{n}.
\end{equation}
Since estimates of $\eta^{n}$ are known from Lemmas \ref{lem5} and \ref{lem6}, it is sufficient to bound $\zeta^{n}$, then use the triangle inequality to bound $u^n-U_h^{n}$.

{F}rom (\ref{eq:1}) and (\ref{eq:3-0}), it follows that, for $ 1\leq n\leq M$,
\begin{eqnarray}
\label{eq5.7}
\lefteqn{\frac{\Delta t^{1-\alpha}}{\Gamma (3-\alpha)}[b_0\delta_t u^{n}-\sum\limits_{j=1}^{n-1}(b_{n-j-1}-b_{n-j})\delta_t u^{j}-b_{n-1}\phi]}
\\
&& \quad
+R_{\alpha}^{n-\frac{1}{2}}
+\frac{\Gamma(3-\alpha)\Delta t^{1+\alpha}}{4}\frac{\partial^4\delta_t u^{n}}{\partial x^2\partial y^2}
\nonumber
\\ &=&\Delta u^{n-\frac{1}{2}}+\frac{\Gamma(3-\alpha)\Delta t^{1+\alpha}}{4}\frac{\partial^4\delta_t u^{n}}{\partial x^2\partial y^2}
+f(t_{n-\frac{1}{2}}), \quad {\rm on}\ \Lambda.
\nonumber
\end{eqnarray}
Combining (\ref{eq5.7}), (\ref{eq4.3}), (\ref{eq:3-6}), (\ref{eq:W}) and (\ref{eq5.5}), we have
\begin{eqnarray}
\label{eq5.8}
\lefteqn{\frac{\Delta t^{1-\alpha}}{\Gamma (3-\alpha)}\delta_t \zeta^{n}-\Delta \zeta^{n-\frac{1}{2}}
+\frac{\Gamma(3-\alpha)\Delta t^{1+\alpha}}{4}\frac{\partial^4\delta_t \zeta^{n}}{\partial x^2\partial y^2}}
\\&=&\frac{\Delta t^{1-\alpha}}{\Gamma (3-\alpha)}\sum\limits_{j=1}^{n-1}(b_{n-j-1}-b_{n-j})\delta_t \zeta^{j}
+\frac{\Delta t^{1-\alpha}}{\Gamma (3-\alpha)}b_{n-1}\delta_t \eta^{1}+{\cal F}_u^{n-\frac{1}{2}} ,\nonumber
\end{eqnarray}
where
\begin{eqnarray}
\label{eq5.9}
{\cal F}_u^{n-\frac{1}{2}} &=&
R_{\alpha}^{n-\frac{1}{2}}
-\frac{\Gamma(3-\alpha)\Delta t^{1+\alpha}}{4}\frac{\partial^4\delta_t u^{n}}{\partial x^2\partial y^2}+\frac{\Gamma(3-\alpha)\Delta t^{1+\alpha}}{4}\frac{\partial^4\delta_t \eta^{n}}{\partial x^2\partial y^2}
\\
&& \quad
+\frac{\Delta t^{1-\alpha}}{\Gamma (3-\alpha)}\sum\limits_{j=1}^{n-1}b_{n-j-1}\left (\delta_t \eta^{j+1}-\delta_t \eta^{j}\right ).
\nonumber
\end{eqnarray}
We first prove (\ref{eq:convergence}) for $j=0,1$.
Applying the stability result (\ref{eq:stability}) of Theorem \ref{th:stability} to  (\ref{eq5.8}), we obtain
\begin{equation}
\label{eq5.10}
\left\|\nabla \zeta^n\right\|^2\leq C \left\|\nabla \zeta^0\right\|^2+\frac{t_n^{2-\alpha}\left\|\delta_t \eta^{1}\right\|_D^2}{\Gamma(3-\alpha)}+\Delta t \sum\limits_{j=1}^{n}\Gamma(2-\alpha)t_n^{\alpha-1}\left\|{\cal F}_u^{j-\frac{1}{2}}\right\|_D^2
\end{equation}
>From (\ref{eq5.9}), we have
\begin{eqnarray}
\label{eq5.11}
\left\|{\cal F}_u^{n-\frac{1}{2}}\right\|_D&\leq&
\left\|R_{\alpha}^{n-\frac{1}{2}}\right\|_D
+\frac{\Gamma(3-\alpha)\Delta t^{1+\alpha}}{4}\left\|\frac{\partial^4\delta_t \eta^{n}}{\partial x^2\partial y^2}-\frac{\partial^4\delta_t u^{n}}{\partial x^2\partial y^2}\right\|_D
\\
&&+\frac{\Delta t^{1-\alpha}}{\Gamma (3-\alpha)}\left\|\sum\limits_{j=1}^{n-1}b_{n-j-1}\left (\delta_t \eta^{j+1}-\delta_t \eta^{j}\right )\right\|_D.
\nonumber
\end{eqnarray}
>From (\ref{eq:3-2}),
\begin{equation}
\label{eq5.12}
\left \|R_{\alpha}^{n-\frac{1}{2}}\right\|_D\leq C{\Delta t}^{3-\alpha},\ \ 1\leq n\leq M .
\end{equation}
For the second term on the right-hand side of (\ref{eq5.11}),  we obtain
\begin{eqnarray}
\label{eq5.13}
\lefteqn{\frac{\Gamma(3-\alpha)\Delta t^{1+\alpha}}{4}\left\|\frac{\partial^4\delta_t \eta^{n}}{\partial x^2\partial y^2}-\frac{\partial^4\delta_t u^{n}}{\partial x^2\partial y^2}\right\|_D}
\\
&\leq&C\Delta t^{1+\alpha}\left[\left\|\frac{\partial u}{\partial t}\right\|_{C^{2,2,0}}+\left\|\frac{\partial u}{\partial t}\right\|_{C\left([0,T],H^{r+3}\right)}\right],
\nonumber
\end{eqnarray}
from \cite[Eqs. (2.58), (2.59)]{fairweather2010}.

The last term on the right-hand side in (\ref{eq5.11}) is bounded in the following way.
First,
\begin{equation}
\label{eq5.14}
\;\;\;\;\;\;\;\frac{\Delta t^{1-\alpha}}{\Gamma (3-\alpha)}\left\|\sum\limits_{j=1}^{n-1}b_{n-j-1}\left (\delta_t \eta^{j+1}-\delta_t \eta^{j}\right )\right\|_D
\leq
\frac{\Delta t^{2-\alpha}}{\Gamma (3-\alpha)}\sum\limits_{j=1}^{n-1}b_{n-j-1}\left\|\delta_t^2 \eta^{j+1}\right\|_D,	
\end{equation}
and
\begin{eqnarray}
\label{eq5.15}
\;\;\;\;\;\;
\left\|\delta^2_t \eta^{j+1}\right\|_D
&=&\frac{1}{\Delta t^2}\left\|\int_{t_{j-1}}^{t_j}(\tau-t_{j-1})\frac{\partial^2\eta}{\partial t^2}(\tau)d\tau-\int_{t_{j}}^{t_{j+1}}(\tau-t_{j+1})\frac{\partial^2\eta}{\partial t^2}(\tau)d\tau\right\|_D
\\
&\leq&\frac{1}{\Delta t}\left[\int_{t_{j-1}}^{t_j}\left\|\frac{\partial^2\eta}{\partial t^2}(\tau)\right\|_D d\tau+\int_{t_{j}}^{t_{j+1}}\left\|\frac{\partial^2\eta}{\partial t^2}(\tau)\right\|_Dd\tau\right]
\nonumber \\
&\leq&Ch^{r+1}\left\|\frac{\partial^2u}{\partial t^2}\right\|_{C([0,T],H^{r+3})},\quad j\geq 1,
\nonumber
\end{eqnarray}
using Lemma \ref{lem6}.
Hence, on substituting (\ref{eq5.15}) into (\ref{eq5.14}), and using $\sum\limits_{j=1}^{n-1}b_{n-j-1}=(n-1)^{2-\alpha}$ from Lemma \ref{lem:1}$(iv)$, we have
\begin{eqnarray}
\label{eq5.17}
\;\;\;\;\;\;\;\;\frac{\Delta t^{1-\alpha}}{\Gamma (3-\alpha)}\left\|\sum\limits_{j=1}^{n-1}b_{n-j-1}\left (\delta_t \eta^{j+1}-\delta_t \eta^{j}\right )\right\|_D
&\leq& C_1h^{r+1}
\frac{t_{n-1}^{2-\alpha}}{\Gamma (3-\alpha)}
\left\|\frac{\partial^2u}{\partial t^2}\right\|_{C([0,T],H^{r+3})}
\\\nonumber\\
&\le& CT^{2-\alpha}h^{r+1}.
\nonumber
\end{eqnarray}
Since $1<\alpha<2$, we have  $1+\alpha>3-\alpha$.
Therefore, with (\ref{eq5.12}), (\ref{eq5.13}) and (\ref{eq5.17}) in (\ref{eq5.11}),
 we obtain
\begin{eqnarray}\label{eq5.18}
\left\|{\cal F}_u^{n-\frac{1}{2}}\right\|_D\leq C\left({\Delta t}^{3-\alpha}+h^{r+1}\right).
\end{eqnarray}
Also,
\begin{eqnarray}
\label{eq5.16}
\left\|\delta_t \eta^{1}\right\|_D=\frac{1}{\Delta t}\left\|\int_{t_{0}}^{t_1}\frac{\partial\eta}{\partial t}(\tau)d\tau\right\|_D
&\leq&
\frac{1}{\Delta t}\int_{t_{0}}^{t_1}\left\|\frac{\partial\eta}{\partial t}(\tau)\right\|_D d\tau
\\\nonumber\\
&\leq& Ch^{r+1}\left\|\frac{\partial u}{\partial t}\right\|_{C\left([0,T],H^{r+3}\right)},\nonumber
\end{eqnarray}
using (\ref{eq5.3}).
Then, substituting (\ref{eq5.18}) and (\ref{eq5.16}) into (\ref{eq5.10}), and noting that $\zeta^0=0$, we have
\begin{eqnarray}
\label{eq5.19}
\left\|\nabla \zeta^n\right\|\leq C\left({\Delta t}^{3-\alpha}+h^{r+1}\right).
\end{eqnarray}
Using Poincar\'{e}'s inequality, it follows that
\begin{eqnarray}
\label{eq5.20}
\left\|\zeta^n\right\|\leq C\left({\Delta t}^{3-\alpha}+h^{r+1}\right).
\end{eqnarray}
Then using the triangle inequality, (\ref{eq5.19}), (\ref{eq5.20}) and Lemma \ref{lem5}, it follows that
\[
\left\|u(t_{n})-U_h^{n}\right\|_{H^{j}}\leq C\left (h^{r+1-j}+\Delta t^{3-\alpha}\right ), \qquad j=0, 1,
\]
as desired.

When $j=2$, according to Theorem \ref{th:stability2}, we obtain
\begin{eqnarray}
\label{eq5.21}
\left\|\zeta^n\right\|_{H^2}^2&\leq& C \left(\left\|\Delta \zeta^0\right\|_D^2+t_q^{2-2\alpha}\left\|\delta_t \eta^{1}\right\|_D^2
\right)
\\
& &
+C \left(\left\| {\cal F}_u^{\frac{1}{2}}\right\|_D^2+\left\|{\cal F}_u^{n-\frac{1}{2}}\right\|_D^2+\Delta t \sum\limits_{j=2}^{n}\left\|\delta_t {\cal F}_u^{j-\frac{1}{2}}\right\|_D^2\right).
\nonumber
\end{eqnarray}
In order to complete the proof,
we require an estimate of $\left\|\delta_t {\cal F}_u^{n-\frac{1}{2}}\right\|_D^2$.
First observe that, from (\ref{eq5.9}), for $2\leq n\leq M$,
\begin{eqnarray}
\label{eq5.22}
\delta_t {\cal F}_u^{n-\frac{1}{2}}&=&
\delta_t R_{\alpha}^{n-\frac{1}{2}}
-\frac{\Gamma(3-\alpha)\Delta t^{1+\alpha}}{4}\frac{\partial^4\delta^2_t u^{n}}{\partial x^2\partial y^2}+\frac{\Gamma(3-\alpha)\Delta t^{1+\alpha}}{4}\frac{\partial^4\delta^2_t \eta^{n}}{\partial x^2\partial y^2}
\\
&&
+\frac{\Delta t^{1-\alpha}}{\Gamma (3-\alpha)}\left[\sum\limits_{j=1}^{n-2}b_{j-1}\left(\delta^2_t \eta^{n-j+1}-\delta^2_t\eta^{n-j}\right)+b_{n-2}\delta^2_t \eta^{2}\right],
\nonumber
\end{eqnarray}
from which it follows that
\begin{eqnarray}
\label{eq5.23}
\;\;\;\;\left\|\delta_t{\cal F}_u^{n-\frac{1}{2}}\right\|_D&\leq&
\left\|\delta_tR_{\alpha}^{n-\frac{1}{2}}\right\|_D
+\frac{\Gamma(3-\alpha)\Delta t^{1+\alpha}}{4}\left\|\frac{\partial^4\delta^2_t \eta^{n}}{\partial x^2\partial y^2}-\frac{\partial^4\delta^2_t u^{n}}{\partial x^2\partial y^2}\right\|_D
\\\nonumber\\
&&
+\frac{\Delta t^{2-\alpha}}{\Gamma (3-\alpha)}\left\|\sum\limits_{j=1}^{n-2}b_{j-1}\delta^3_t \eta^{n-j+1}\right\|_D+\frac{\Delta t^{1-\alpha}}{\Gamma (3-\alpha)}b_{n-2}\left\|\delta^2_t \eta^{2}\right\|_D.
\nonumber
\end{eqnarray}
In Appendix $B$, it is proved that
\begin{equation}
\label{eq5.24}
\left \|\delta_t R_{\alpha}^{n-\frac{1}{2}}\right\|_D\leq C{\Delta t}^{3-\alpha},\ \ 1\leq n\leq M.
\end{equation}

Then
\begin{eqnarray}
\label{eq5.25}
\lefteqn{\frac{\Gamma(3-\alpha)\Delta t^{1+\alpha}}{4}\left\|\frac{\partial^4\delta^2_t \eta^{n}}{\partial x^2\partial y^2}-\frac{\partial^4\delta^2_t u^{n}}{\partial x^2\partial y^2}\right\|_D}
\\
&\leq&C\Delta t^{1+\alpha}\left[\left\|\frac{\partial^2 u}{\partial t^2}\right\|_{C^{2,2,0}}+\left\|\frac{\partial^2 u}{\partial t^2}\right\|_{C\left([0,T],H^{r+3}\right)}\right];
\nonumber
\end{eqnarray}
cf., (\ref{eq5.13}).
Also, for $j\geq 3$,
\begin{eqnarray}
\label{eq5.27}
\left\|\delta^3_t \eta^{j}\right\|_D
&=&\frac{1}{\Delta t^3}\left\|\int_{t_{j-1}}^{t_j}\frac{1}{2}(s-t_{j})^2\frac{\partial^3\eta}{\partial t^3}(\cdot,s)ds-\int_{t_{j-2}}^{t_{j-1}}\frac{1}{2}(s-t_{j-2})^2\frac{\partial^3\eta}{\partial t^3}(\cdot,s)ds \right .\nonumber
\\
&&
\left .+\int_{t_{j-3}}^{t_{j-2}}\frac{1}{2}(s-t_{j-3})^2\frac{\partial^3\eta}{\partial t^3}(\cdot,s)ds-\int_{t_{j-2}}^{t_{j-1}}\frac{1}{2}(s-t_{j-1})^2\frac{\partial^3\eta}{\partial t^3}(\cdot,s)ds\right\|_D
\\
&\leq&\frac{1}{\Delta t}\left[\frac{1}{2}\int_{t_{j-1}}^{t_j}\left\|\frac{\partial^3\eta}{\partial t^3}\right\|_D ds+\int_{t_{j-2}}^{t_{j-1}}\left\|\frac{\partial^3\eta}{\partial t^3}\right\|_Dds
+\frac{1}{2}\int_{t_{j-3}}^{t_{j-2}}\left\|\frac{\partial^3\eta}{\partial t^3}\right\|_Dds\right]
\nonumber\\
&\leq&
 Ch^{r+1}\left\|\frac{\partial^3u}{\partial t^3}\right\|_{C([0,T],H^{r+3})},
\nonumber
\end{eqnarray}
using Lemma \ref{lem6}.
Thus,
since $ \sum\limits_{j=1}^{n-2}b_{j-1}=(n-2)^{2-\alpha}$ from Lemma \ref{lem:1}$(iv)$,  we obtain, on using (\ref{eq5.27}),
\begin{eqnarray}
\label{eq5.26}
\frac{\Delta t^{2-\alpha}}{\Gamma (3-\alpha)}\left\|\sum\limits_{j=1}^{n-2}b_{j-1}\delta^3_t \eta^{n-j+1}\right\|_D
&\leq&
\frac{\Delta t^{2-\alpha}}{\Gamma (3-\alpha)}\sum\limits_{j=1}^{n-2}b_{j-1}\left\|\delta^3_t \eta^{n-j+1}\right\|_D
\\
&\leq&
\frac{Ch^{r+1}t_{n-2}^{2-\alpha}}{\Gamma (3-\alpha)}\left\|\frac{\partial^3u}{\partial t^3}\right\|_{C([0,T],H^{r+3})},
\nonumber
\end{eqnarray}
for $n\geq 3$,
and, from (\ref{eq5.15}),
\begin{eqnarray}
\label{eq5.26a}
\frac{\Delta t^{1-\alpha}}{\Gamma (3-\alpha)}b_{n-2}\left\|\delta^2_t \eta^{2}\right\|_D \leq
\frac{Ch^{r+1}t_{n-2}^{1-\alpha}}{\Gamma (2-\alpha)}\left\|\frac{\partial^2u}{\partial t^2}\right\|_{C([0,T],H^{r+3})},
\end{eqnarray}
since $ b_j<(2-\alpha)j^{1-\alpha}$ from Lemma \ref{lem:1}(ii).
The estimate (\ref{eq5.26a}) also holds for $n=2$, since $b_0=1$, $t_1=\Delta t$.

Since $1<\alpha<2$, we have  $1+\alpha>3-\alpha$.
Therefore, substituting (\ref{eq5.24}), (\ref{eq5.25}), (\ref{eq5.26}) and (\ref{eq5.26a}) in (\ref{eq5.23}),
 we obtain
\begin{eqnarray}\label{eq5.29}
\left\|\delta_t {\cal F}_u^{n-\frac{1}{2}}\right\|_D\leq C\left({\Delta t}^{3-\alpha}+h^{r+1}\right).
\end{eqnarray}
Also, it follows from (\ref{eq5.18}), (\ref{eq5.16}), and (\ref{eq5.29}) that
\begin{eqnarray}
\label{eq5.30}
\left\|\zeta^n\right\|_{H^2}^2\leq C\left({\Delta t}^{3-\alpha}+h^{r+1}\right),
\end{eqnarray}
since $\zeta^0=0$. Finally, applying the triangle inequality, (\ref{eq5.30}) and Lemma \ref{lem5}  complete the proof.
\qquad \end{proof}

{\textit{Remark.}}
Note that results similar to those in \cite[Remark 2.7]{fairweather2010} are also valid in this paper.

\section{Numerical experiments}
\setcounter{equation}{0}
We present numerical results which support the analyses of preceding sections.
In our implementations, we used
the space of piecewise Hermite bicubics with the standard basis functions \cite{fairweather1978} on identical
uniform partitions of $\overline{I}$ in both the $x$ and $y$
directions with $N_x=N_y=N$. The initial conditions are
approximated using the piecewise Hermite bicubic interpolant.
We present $L^{\infty}$, $L^{2}$, $H^{1}$ and $H^{2}$ norms of the errors at $T=1$ and the
corresponding rates of convergence determined by
\begin{equation}\label{eq:45}
\mbox{Rate} \approx\frac{\log(e_m/e_{m+1})}{\log(h_m/h_{m+1})},
\end{equation}
where $h=1/N_m$ is the step size with $N=N_m$, and $e_m$ is the norm
of the corresponding error.
The $L^{\infty}$ norm of the error is estimated by calculating the maximum error at $100\times100$ equally spaced points in each
sub-rectangle $[x_{i-1},x_i]\times [y_{j-1},y_j]$, $1\leq i,j\leq N$.
The $H^{\ell}$ norm, $\ell=0,1,2$, is computed using the ten-point composite Gauss quadrature rule so that the error due to quadrature does
not affect the convergence rate.

{\bf Example. \cite{sunzhizhong2012}} We consider the problem                                                                                                                                                (\ref{eq:1})--(\ref{eq:3}) with $T=1$ and exact solution
\[
u(x,y,t)=t^{2+\alpha}\sin(\pi x)\sin(\pi y), \quad (x,y)\in[0,1]\times [0,1], \quad t\in(0,T],
\]
so that
\[
\varphi(x,y)=0, \quad \phi(x,y)=0,\quad
(x,y)\in[0,1]\times [0,1].
\]

\begin{table}[h]
\centering \caption{$L^2$ and $L^{\infty}$ errors and convergence rates  with $\Delta t=h^3$, $\alpha=1.5$.} \label{lab:example1-1}
\begin{tabular}{|c|c|c|c|c|c|c|c|c|}
\hline\ \ \  \ \ \  $N$\ \ \  \ \ \    & $L^{2}$ error & \ \ \  \ \ \ Rate\ \ \  \ \ \     & $L^{\infty}$ error       &\ \ \  \ \ \  Rate\ \ \  \ \ \   \\
\hline $4$   & 3.1505e-004  &    & 1.1250e-3 &        \\
\hline $6$   &5.3158-5 & 4.3887& 2.0968e-4 &   4.1433 \\
\hline $9$  & 1.0142e-5  &4.0856  &4.2644e-5 &3.9281\\
\hline $12$  & 3.1984e-5  &4.0115      & 1.3239e-5 &  4.0660   \\
\hline
\end{tabular}
\end{table}

\begin{table}[h]
\centering \caption{$H^1$ and $H^2$ errors and convergence rates with $\Delta t=h^3$, $\alpha=1.5$.} \label{lab:example1-2}
\begin{tabular}{|c|c|c|c|c|c|c|c|c|}
\hline\ \ \  \ \ \  $N$\ \ \  \ \ \    & $H^{1}$ error & \ \ \  \ \ \ Rate\ \ \  \ \ \     & $H^{2}$ error       &\ \ \  \ \ \  Rate\ \ \  \ \ \   \\
\hline $4$   & 6.2642e-3  &    & 1.6180e-1 &        \\
\hline $6$   &1.8393-3 & 3.0224& 7.1581e-2 &   2.0113 \\
\hline $9$  & 5.4400e-4  &3.0044  &3.1744e-2 &2.0054\\
\hline $12$  & 2.2938e-4  &3.0018      & 1.7844e-2 &  2.0023   \\
\hline
\end{tabular}
\end{table}

\begin{table}[h]
\centering \caption{Maximum nodal errors in  $(U^M_h)_x,(U^M_h)_y$ and convergence rates  with $\Delta t=h^3$, $\alpha=1.5$.} \label{lab:example1-4}
\begin{tabular}{|c|c|c|c|c|c|c|c|c|}
\hline $N$ &  Maximum nodal error in $(U^M_h)_x,(U^M_h)_y$ & Rate \\
\hline $4$   & 1.1613e-3  & \\        
\hline $6$   &3.2174e-4& 3.1656\\
\hline $9$  & 6.6230e-5  &3.8983\\
\hline $12$  & 2.1151e-5  & 3.9677\\  
\hline
\end{tabular}
\end{table}

\begin{table}[h]
\centering \caption{$L^2$ and $L^{\infty}$ errors and convergence rates   with $\Delta t=h$, $\alpha=1.1$.} \label{lab:example1-5}
\begin{tabular}{|c|c|c|c|c|c|c|c|c|}
\hline\ \ \  \ \ \  $N$\ \ \  \ \ \    & $L^{2}$ error & \ \ \  \ \ \ Rate\ \ \  \ \ \     & $L^{\infty}$ error       &\ \ \  \ \ \  Rate\ \ \  \ \ \   \\
\hline $40$   &1.3847e-6 &  & 2.7240e-6 &    \\
\hline $80$  &  4.4740e-7  &1.6299  & 6.8463e-7 & 1.9923\\
\hline $160$  &  1.3113e-7  &1.7706      & 1.8874e-7 &  1.8589   \\
\hline $320$  &  3.7282e-8  &1.8144      &   5.2931e-8 &  1.8342   \\
\hline
\end{tabular}
\end{table}

\begin{table}[h]
\centering \caption{$L^2$ and $L^{\infty}$ errors and convergence rates  with $\Delta t=h$, $\alpha=1.45$.} \label{lab:example1-6}
\begin{tabular}{|c|c|c|c|c|c|c|c|c|}
\hline\ \ \  \ \ \  $N$\ \ \  \ \ \    & $L^{2}$ error & \ \ \  \ \ \ Rate\ \ \  \ \ \     & $L^{\infty}$ error       &\ \ \  \ \ \  Rate\ \ \  \ \ \   \\
\hline $20$   & 6.3816e-5  &    &  1.0336e-4 &        \\
\hline $40$   &2.3412e-5 &  1.4467& 3.3951e-5 &    1.6062 \\
\hline $80$  & 8.1811e-6  & 1.5169  &1.1623e-5 & 1.5465\\
\hline $160$  & 2.8253e-6  & 1.5339      & 3.9989e-6 &  1.5393   \\
\hline
\end{tabular}
\end{table}

\begin{table}[h]
\centering \caption{$L^2$ and $L^{\infty}$ errors and convergence rates  with $\Delta t=h$, $\alpha=1.8$.} \label{lab:example1-7}
\begin{tabular}{|c|c|c|c|c|c|c|c|c|}
\hline\ \ \  \ \ \  $N$\ \ \  \ \ \    & $L^{2}$ error & \ \ \  \ \ \ Rate\ \ \  \ \ \     & $L^{\infty}$ error       &\ \ \  \ \ \  Rate\ \ \  \ \ \   \\
\hline $20$   & 3.6625e-4  &    & 5.3143e-4 &        \\
\hline $40$   &1.6144e-4 &  1.1818& 2.2916e-4 &    1.2135 \\
\hline $80$  & 7.0458e-5  & 1.1962  &9.9696e-5 & 1.2007\\
\hline $160$  & 3.0695e-5  & 1.1988      & 4.3413e-5 &  1.1994  \\
\hline
\end{tabular}
\end{table}

For the results in Tables \ref{lab:example1-1}--\ref{lab:example1-4}, we selected the time step $\Delta t=h^3$, since, from our theoretical estimates, the error in the $H^{\ell}$ norm, $\ell=0,1,2$, is expected to be $O(\Delta t ^{3-\alpha}+h^{4-\ell})$ when $r=3$.
In Table \ref{lab:example1-1}, we present the $L^2$ and $L^{\infty}$ errors with their corresponding convergence rates which are seen to be approximately $4$ as expected. Table \ref{lab:example1-2} demonstrates the optimal convergence rates in the $H^{\ell}, \ell=1,2$, norms, consistent with the theory.
In Table \ref{lab:example1-4},
 we present the maximum error in the approximations $(U^M_h)_x,(U^M_h)_y$ to $u_x,\;u_y$, respectively, at the partition nodes, together with the corresponding convergence rate. From this table, we  observe superconvergence, the rate of convergence being approximately $4$ when $r=3$. In Tables \ref{lab:example1-5}--\ref{lab:example1-7}, we present the $L^2$ and $L^{\infty}$ errors and the temporal convergence rates, which is approximately $(3-\alpha)$ as expected.

\section{Concluding Remarks}
\setcounter{equation}{0}
We have formulated and analyzed an ADI OSC Crank-Nicolson method for the two-dimensional fractional diffusion-wave equation. Under certain smoothness assumptions, we have proved that the method is of optimal global accuracy and exhibits superconvergence phenomena. The results of numerical experiments confirm the analysis.

\renewcommand{\theequation}{A.\arabic{equation}}
\setcounter{equation}{0}  
\section*{APPENDIX A: Proof of (\ref{eq4.32})}

\vspace{.25in}
\noindent
\begin{lemma}
There exist non-negative integers $m,q$,  $1\leq m\leq n$ and $0\leq q\leq n-1$, such that
\begin{eqnarray}
\label{eqA4.1}
\sum\limits_{j=1}^{n}b_{j-1}\left\langle\phi,\Delta U_h^{j}-\Delta U_h^{j-1}\right\rangle
\leq
b_{q}\left\langle\phi,\Delta U_h^{m}-\Delta U_h^{0}\right\rangle,
\end{eqnarray}
\end{lemma}
\begin{proof}
We  prove the result by induction.
When $n=1$,
\begin{eqnarray}
\label{eqA4.2}
b_{0}\left\langle\phi,\Delta U_h^{1}-\Delta U_h^{0}\right\rangle
\leq
b_{q}\left\langle\phi,\Delta U_h^{m}-\Delta U_h^{0}\right\rangle,
\end{eqnarray}
and $q=0, m=1$. Therefore, (\ref{eqA4.1}) holds for the case $n=1$.

Now suppose that
\begin{eqnarray}
\label{eqA4.3}
\sum\limits_{j=1}^{n}b_{j-1}\left\langle\phi,\Delta U_h^{j}-\Delta U_h^{j-1}\right\rangle
\leq
b_{q}\left\langle\phi,\Delta U_h^{m}-\Delta U_h^{0}\right\rangle,&&
\\
 1\leq m\leq n, \quad&&0\leq q\leq n-1, \nonumber
\end{eqnarray}
Then
\begin{eqnarray}
\label{eqA4.5}
\lefteqn{\sum\limits_{j=1}^{n+1}b_{j-1}\left\langle\phi,\Delta U_h^{j}-\Delta U_h^{j-1}\right\rangle}
\\ &
=&\sum\limits_{j=1}^{n}b_{j-1}\left\langle\phi,\Delta U_h^{j}-\Delta U_h^{j-1}\right\rangle+b_{n}\left\langle\phi,\Delta U_h^{n+1}-\Delta U_h^{n}\right\rangle.
\nonumber
\\
\nonumber
&\leq&
b_{q}\left\langle\phi,\Delta U_h^{m}-\Delta U_h^{0}\right\rangle+b_{n}\left\langle\phi,\Delta U_h^{n+1}-\Delta U_h^{n}\right\rangle.
\nonumber
\end{eqnarray}
 on using the induction hypothesis (\ref{eqA4.3}).

\begin{enumerate}
\item First consider the case $m=n$.

\vspace{.2in}
\begin{enumerate}
\item If $\left\langle\phi,\Delta U_h^{n}-\Delta U_h^{0}\right\rangle\geq 0$ and $\left\langle\phi,\Delta U_h^{n+1}-\Delta U_h^{n}\right\rangle \geq 0$, then
\begin{eqnarray*}
\label{eqA4.7}
\lefteqn{
b_{q}\left\langle\phi,\Delta U_h^{n}-\Delta U_h^{0}\right\rangle+b_{n}\left\langle\phi,\Delta U_h^{n+1}-\Delta U_h^{n}\right\rangle}
\\\\
&\leq&b_{q}\left\langle\phi,\Delta U_h^{n}-\Delta U_h^{0}\right\rangle+b_{q}\left\langle\phi,\Delta U_h^{n+1}-\Delta U_h^{n}\right\rangle
\nonumber  \\\\&=& b_{q}\left\langle\phi,\Delta U_h^{n+1}-\Delta U_h^{0}\right\rangle
,\quad  0\leq q\leq n-1.
\nonumber
\end{eqnarray*}

\medskip
\item If $\left\langle\phi,\Delta U_h^{n}-\Delta U_h^{0}\right\rangle\geq 0$ and $\left\langle\phi,\Delta U_h^{n+1}-\Delta U_h^{n}\right\rangle \leq 0$, then
\begin{eqnarray*}
\label{eqA4.8}
\lefteqn{
b_{q}\left\langle\phi,\Delta U_h^{n}-\Delta U_h^{0}\right\rangle+b_{n}\left\langle\phi,\Delta U_h^{n+1}-\Delta U_h^{n}\right\rangle}
\\
\nonumber
\\
&\leq& b_{q}\left\langle\phi,\Delta U_h^{n}-\Delta U_h^{0}\right\rangle,\quad 0\leq q\leq n-1.
\nonumber
\end{eqnarray*}

\medskip
\item If $\left\langle\phi,\Delta U_h^{n}-\Delta U_h^{0}\right\rangle\leq 0$ and $\left\langle\phi,\Delta U_h^{n+1}-\Delta U_h^{n}\right\rangle \geq 0$, then
\begin{eqnarray*}
\label{eqA4.9}
\lefteqn{
b_{q}\left\langle\phi,\Delta U_h^{n}-\Delta U_h^{0}\right\rangle+b_{n}\left\langle\phi,\Delta U_h^{n+1}-\Delta U_h^{n}\right\rangle}
\\\\ &
\leq& b_{n}\left\langle\phi,\Delta U_h^{n}-\Delta U_h^{0}\right\rangle+b_{n}\left\langle\phi,\Delta U_h^{n+1}-\Delta U_h^{n}\right\rangle
\nonumber  \\\\&=&  b_{n}\left\langle\phi,\Delta U_h^{n+1}-\Delta U_h^{0}\right\rangle.
\nonumber
\end{eqnarray*}

\medskip
\item
 If $\left\langle\phi,\Delta U_h^{n}-\Delta U_h^{0}\right\rangle\leq 0$ and $\left\langle\phi,\Delta U_h^{n+1}-\Delta U_h^{n}\right\rangle \leq 0$, then
\begin{eqnarray*}
\label{eqA4.10}
\lefteqn{
b_{q}\left\langle\phi,\Delta U_h^{n}-\Delta U_h^{0}\right\rangle+b_{n}\left\langle\phi,\Delta U_h^{n+1}-\Delta U_h^{n}\right\rangle}
\\ \\&
\leq& b_{n}\left\langle\phi,\Delta U_h^{n}-\Delta U_h^{0}\right\rangle+b_{n}\left\langle\phi,\Delta U_h^{n+1}-\Delta U_h^{n}\right\rangle
\nonumber  \\\\&=&b_{n}\left\langle\phi,\Delta U_h^{n+1}-\Delta U_h^{0}\right\rangle.
\nonumber
\end{eqnarray*}
\end{enumerate}
Therefore, in this case, (\ref{eqA4.1}) holds.

\vspace{.2in}
\item  Now consider the case $m < n$. Recall from (\ref{eqA4.5}) that
\begin{eqnarray}
\label{eqA4.11}
\lefteqn{\sum\limits_{j=1}^{n+1}b_{j-1}\left\langle\phi,\Delta U_h^{j}-\Delta U_h^{j-1}\right\rangle}
\\ &
=&\sum\limits_{j=1}^{n}b_{j-1}\left\langle\phi,\Delta U_h^{j}-\Delta U_h^{j-1}\right\rangle+b_{n}\left\langle\phi,\Delta U_h^{n+1}-\Delta U_h^{n}\right\rangle.\nonumber
\end{eqnarray}

\medskip
\begin{enumerate}
\item
If $\left\langle\phi,\Delta U_h^{n+1}-\Delta U_h^{n}\right\rangle\leq 0$, then, by using the induction hypothesis (\ref{eqA4.3}),
\begin{eqnarray*}
\label{eqA4.12}
\lefteqn{\sum\limits_{j=1}^{n+1}b_{j-1}\left\langle\phi,\Delta U_h^{j}-\Delta U_h^{j-1}\right\rangle
\leq
 b_{q}\left\langle\phi,\Delta U_h^{m}-\Delta U_h^{0}\right\rangle,}\\&& \hspace{2.in} 1\leq m\leq n, \quad  0\leq q\leq n-1.
\nonumber
\end{eqnarray*}

\medskip
\item If $\left\langle\phi,\Delta U_h^{n+1}-\Delta U_h^{n}\right\rangle\geq 0$, then, from (\ref{eqA4.11}),
\begin{eqnarray*}
\label{eqA4.13}
\lefteqn{\sum\limits_{j=1}^{n+1}b_{j-1}\left\langle\phi,\Delta U_h^{j}-\Delta U_h^{j-1}\right\rangle}
\\\\ &&
\leq
\sum\limits_{j=1}^{n-1}b_{j-1}\left\langle\phi,\Delta U_h^{j}-\Delta U_h^{j-1}\right\rangle+b_{n-1}\left\langle\phi,\Delta U_h^{n+1}-\Delta U_h^{n-1}\right\rangle,
\nonumber
\end{eqnarray*}
since $b_n < b_{n-1}$.

\medskip
\begin{enumerate}
\item If $\left\langle\phi,\Delta U_h^{n+1}-\Delta U_h^{n-1}\right\rangle\leq 0$, then, by using the induction hypothesis (\ref{eqA4.3}),
\begin{eqnarray*}
\label{eqA4.14}
\lefteqn{\sum\limits_{j=1}^{n+1}b_{j-1}\left\langle\phi,\Delta U_h^{j}-\Delta U_h^{j-1}\right\rangle}
\\ \\&
\leq&
b_{q}\left\langle\phi,\Delta U_h^{m}-\Delta U_h^{0}\right\rangle, 1\leq m\leq n-1, 0\leq q\leq n-2.
\nonumber
\end{eqnarray*}

\item If $\left\langle\phi,\Delta U_h^{n+1}-\Delta U_h^{n-1}\right\rangle\geq 0$, from (\ref{eqA4.13}), then
\begin{eqnarray*}
\label{eqA4.15}
\lefteqn{\sum\limits_{j=1}^{n+1}b_{j-1}\left\langle\phi,\Delta U_h^{j}-\Delta U_h^{j-1}\right\rangle}
\\\\ &
\leq&
\sum\limits_{j=1}^{n-2}b_{j-1}\left\langle\phi,\Delta U_h^{j}-\Delta U_h^{j-1}\right\rangle
+b_{n-2}\left\langle\phi,\Delta U_h^{n+1}-\Delta U_h^{n-2}\right\rangle.
\nonumber
\end{eqnarray*}
\end{enumerate}
\end{enumerate}
\end{enumerate}
Repeating the above process, we obtain
\begin{eqnarray*}
\label{eqA4.16}
\lefteqn{\sum\limits_{j=1}^{n+1}b_{j-1}\left\langle\phi,\Delta U_h^{j}-\Delta U_h^{j-1}\right\rangle}
\\\\ &
\leq&
b_{q}\left\langle\phi,\Delta U_h^{m}-\Delta U_h^{0}\right\rangle, \quad 1\leq m\leq n+1, 0\leq q\leq n,
\nonumber
\end{eqnarray*}
which completes the proof of (\ref{eqA4.1}).
\end{proof}

\renewcommand{\theequation}{B.\arabic{equation}}
\setcounter{equation}{0}
\section*{APPENDIX B: Proof of (\ref{eq5.24})}

If the hypotheses of Theorem (\ref{th:convergence}) are satisfied,
we will prove that
\[
\left \|\delta_t R_{\alpha}^{n-\frac{1}{2}}\right\|_D\;\leq\;
C\Delta t^{3-\alpha}.
\]
We set $\tau=\Delta t$, $v=u_t$, and $v^n$ for the approximation to $v(x,y,t_n)$.
From (\ref{eq:1}), we have
\bea
\label{eqB1}
&&\frac{\tau^{1-\alpha}}{\Gamma (3-\alpha)}\left[b_0 v^{n-\frac{1}{2}}-\sum\limits_{j=1}^{n-1}(b_{n-j-1}-b_{n-j}) v^{j-\frac{1}{2}}-b_{n-1}v^0\right]
\\
\nonumber\\
&=&\Delta u^{n-\frac{1}{2}}+f^{n-\frac{1}{2}}-\left(\widetilde{R_t^{\alpha}}\right)^{n-\frac{1}{2}},\nonumber
\eea
where
\beas
\left(\widetilde{R_t^{\alpha}}\right)^{n}
&=&\frac{1}{\Gamma (2-\alpha)}\int_0^{t_{n}}\frac{\partial v(x,y,s)}{\partial s}\frac{ds}{(t_{n}-s)^{\alpha-1}}
\\
\nonumber\\
&&
-\frac{\tau^{1-\alpha}}{\Gamma (3-\alpha)}\left[b_0 v^{n}-\sum\limits_{j=1}^{n-1}(b_{n-j-1}-b_{n-j}) v^{j}-b_{n-1}v^0\right]
\nonumber
\\
\nonumber\\
&=&\frac{1}{\Gamma (2-\alpha)}\int_0^{t_{n}}\frac{\partial v(x,y,s)}{\partial s}\frac{ds}{(t_{n}-s)^{\alpha-1}}
\nonumber
\\
\nonumber
\\
&&
-\frac{\tau^{1-\alpha}}{\Gamma (3-\alpha)}\sum\limits_{j=1}^n\frac{v^{j}-v^{j-1}}{\tau }\int_{t_{j-1}}^{t_{j}}(t_{n}-s)^{1-\alpha}ds
\nonumber
\\
\nonumber
\\
&=&\frac{1}{\Gamma (2-\alpha)}\sum\limits_{j=1}^{n}\int_{t_{j-1}}^{t_j}\left\{\frac{\partial v(x,y,s)}{\partial s}
 -\frac{v^{j}-v^{j-1}}{\tau }\right\}\frac{ds}{(t_{n}-s)^{\alpha-1}}.
 \nonumber
\eeas
Using Taylor's theorem  with integral remainder, and interchanging the order of integration in the resulting equation, we obtain
\bea
\label{eqB7}
\left(\widetilde{R_t^{\alpha}}\right)^{n}
&=&\frac{1}{\Gamma (2-\alpha)\tau}\sum\limits_{j=1}^{n}\int_{t_{j-1}}^{t_j}\left[\int_{t_{j-1}}^{s}\frac{\partial^2 v(x,y,t)}{\partial t^2}(t-t_{j-1})dt
\right.
\\
\nonumber
\\ &\qquad&
\left.
-\int^{t_{j}}_{s}\frac{\partial^2 v(x,y,t)}{\partial t^2}(t_{j}-t)dt
\right]\frac{ds}{(t_{n}-s)^{\alpha-1}}
\nonumber\\\nonumber
\\
&=&\frac{1}{\Gamma (3-\alpha)}\sum\limits_{j=1}^{n}\int_{t_{j-1}}^{t_j}\left\{(t_{n}-s)^{2-\alpha}
-\left[\frac{s-t_{j-1}}{\tau}(t_{n}-t_j)^{2-\alpha}\right.\right.
\nonumber\\ &\qquad&
+\left.\left.\frac{t_{j}-s}{\tau}(t_{n}-t_{j-1})^{2-\alpha}\right]\right\}\frac{\partial^2 v}{\partial s^2}(x,y,s)ds
\nonumber\\
\nonumber
\\
&=&\frac{1}{\Gamma (3-\alpha)}\left(R_t^{\alpha}\right)^{n}.\nonumber
\eea
With
$g(s)=(t_{n}-s)^{2-\alpha}$,
\begin{eqnarray}
\label{eqR-alpha-t-1}
\hspace{.1in}(R_t^{\alpha})^{n}&=&
\sum\limits_{k=1}^{n}\int_{t_{k-1}}^{t_k}\left\{g(s)
-\left[\frac{s-t_{k-1}}{\tau}g(t_k)
+\frac{t_{k}-s}{\tau}g(t_{k-1})\right]\right\}\frac{\partial^2 v}{\partial s^2}(x,y,s)ds
\\ \nonumber\\
&=&
\widehat{r}_{q}+ r_q
\nonumber,
\end{eqnarray}
where
\begin{eqnarray*}
\label{eq10}
&&\\
\widehat{r}_{q}&=&
\sum\limits_{k=1}^{n-q}\int_{t_{k-1}}^{t_k}\left\{g(s)
-\left[\frac{s-t_{k-1}}{\tau}g(t_k)
+\frac{t_{k}-s}{\tau}g(t_{k-1})\right]\right\}\frac{\partial^2 v}{\partial s^2}(x,y,s)ds
\nonumber
\\\\
r_{q}&=&
\sum\limits_{k=n-q+1}^{n}\int_{t_{k-1}}^{t_k}\left\{g(s)
-\left[\frac{s-t_{k-1}}{\tau}g(t_k)
+\frac{t_{k}-s}{\tau}g(t_{k-1})\right]\right\}\frac{\partial^2 v}{\partial s^2}(x,y,s)ds.
\nonumber
\end{eqnarray*}

On applying Lemma $2.1$ in [13], we have
\[
\widehat{r}_{q}
=\frac{(2-\alpha)(\alpha-1)}{2}
\sum\limits_{k=1}^{n-q}(t_{n}-\xi_{k})^{-\alpha}\int_{t_{k-1}}^{t_k}(s-t_{k-1})
(t_{k}-s)\frac{\partial^2 v}{\partial s^2}(x,y,s)ds,
\]
where $\xi_k\in(t_{k-1},t_k)$.
Then
\begin{eqnarray*}
\label{eq17}
&&
\\ &&
\delta_t\widehat{r}_{q}
=\frac{(2-\alpha)(\alpha-1)}{2}
\sum\limits_{k=1}^{n-q-1}\delta_t(t_{n}-\xi_{k})^{-\alpha}\int_{t_{k-1}}^{t_k}(s-t_{k-1})
(t_{k}-s)\frac{\partial^2 v}{\partial s^2}(x,y,s)ds
\nonumber\\ \nonumber\\&&
+\frac{(2-\alpha)(\alpha-1)}{2}
\frac{(t_{n}-\xi_{n-q})^{-\alpha}}{\tau}\int_{t_{n-q-1}}^{t_{n-q}}(s-t_{n-q-1})
(t_{n-q}-s)\frac{\partial^2 v}{\partial s^2}(x,y,s)ds,
\nonumber
\end{eqnarray*}
and since
\[
\int_{t_{k-1}}^{t_k}(s-t_{k-1})
(t_{k}-s)ds=\frac{\tau^3}{6},
\]
we obtain
\begin{eqnarray}
\label{eq19}
\left|\delta_t \widehat{r}_{q}\right|
 &\leq&
C\tau^3\max\limits_{t_0\leq t\leq t_n}\left|\frac{\partial^2 v}{\partial t^2}(x,y,t)\right|
\left|\sum\limits_{k=1}^{n-q-1}(\widehat{\xi}_{n-1}-\xi_{k})^{-\alpha-1}\right|
\\ & &\nonumber\\
&&+C\tau^3\max\limits_{t_0\leq t\leq t_n}\left|\frac{\partial^2 v}{\partial t^2}(x,y,t)\right|
\frac{(t_{n}-t_{n-q})^{-\alpha}}{\tau}
\nonumber\\ \nonumber\\&\leq&
C\tau^3\max\limits_{t_0\leq t\leq t_n}\left|\frac{\partial^2 v}{\partial t^2}(x,y,t)\right|
\left|\sum\limits_{k=1}^{n-q-1}(t_{n-1}-t_{k})^{-\alpha-1}\right|
\nonumber\\ \nonumber\\& &
+C\tau^2t_q^{-\alpha}\max\limits_{t_0\leq t\leq t_n}\left|\frac{\partial^2 v}{\partial t^2}(x,y,t)\right|
\nonumber\\ \nonumber\\&\leq&
C\tau^2t_{q-1}^{-\alpha}\max\limits_{t_0\leq t\leq t_n}\left|\frac{\partial^2 v}{\partial t^2}(x,y,t)\right|,
\nonumber
\end{eqnarray}
where $\widehat{\xi}_{n-1}\in(t_{n-1},t_n)$, $\xi_k\in(t_{k-1},t_k)$.

To estimate $r_q$, we first note that, from Taylor's theorem, we have
\[
\frac{\partial^2 v}{\partial s^2}(x,y,s)=\frac{\partial^2 v}{\partial s^2}(x,y,t_{n})+(s-t_{n})\frac{\partial^3 v}{\partial s^3}(x,y,\overline{\xi}_{n}),\quad \overline{\xi}_{n}\in (s,t_{n}),
\]
so that
\begin{eqnarray*}
\label{eq12}
&&
\\ &&
r_{q}=
\sum\limits_{k=n-q+1}^{n}\int_{t_{k-1}}^{t_k}\left\{g(s)
-\left[\frac{s-t_{k-1}}{\tau}g(t_{k})
+\frac{t_{k}-s}{\tau}g(t_{k-1})\right]\right\}\frac{\partial^2 v}{\partial s^2}(x,y,t_{n})ds
\nonumber\\ &&
+
\sum\limits_{k=n-q+1}^{n}\int_{t_{k-1}}^{t_k}\left\{g(s)
-\left[\frac{s-t_{k-1}}{\tau}g(t_{k})
+\frac{t_{k}-s}{\tau}g(t_{k-1})\right]\right\}(s-t_{n})\frac{\partial^3 v}{\partial s^3}(x,y,\overline{\xi}_{n})ds.
\nonumber
\end{eqnarray*}
Since
\[
\int_{t_{k-1}}^{t_k}\left(\frac{s-t_{k-1}}{\tau}g(t_{k})
+\frac{t_{k}-s}{\tau}g(t_{k-1})\right)ds=\frac{\tau}{2}\left[g(t_k)+g(t_{k-1})\right],
\]
we have
\begin{eqnarray*}
\label{eq14}
&&
\\ &&
r_{q}=\frac{\partial^2 v(x,y,t_{n})}{\partial s^2}\left\{\int_{t_{n-q}}^{t_n}g(s)ds
-\sum\limits_{k=n-q+1}^{n}\frac{\tau}{2}\left[g(t_k)+g(t_{k-1})\right]\right\}
\nonumber\\ \nonumber\\&&
+
\sum\limits_{k=n-q+1}^{n}\int_{t_{k-1}}^{t_k}\left\{g(s)
-\left[\frac{s-t_{k-1}}{\tau}g(t_{k})
+\frac{t_{k}-s}{\tau}g(t_{k-1})\right]\right\}(s-t_{n})\frac{\partial^3 v}{\partial s^3}(x,y,\overline{\xi}_{n})ds
\nonumber\\ \nonumber\\ &&
=\frac{\partial^2 v}{\partial s^2}(x,y,t_{n})\left\{\frac{q^{3-\alpha}}{3-\alpha}
-\left[\frac{q^{2-\alpha}}{2}
+(q-1)^{2-\alpha}+(q-2)^{2-\alpha}+\cdots+1^{2-\alpha}\right]\right\}\tau^{3-\alpha}
\nonumber\\ \nonumber\\&&
+
\sum\limits_{k=n-q+1}^{n}\int_{t_{k-1}}^{t_k}\left\{g(s)
-\left[\frac{s-t_{k-1}}{\tau}g(t_{k})
+\frac{t_{k}-s}{\tau}g(t_{k-1})\right]\right\}(s-t_{n})\frac{\partial^3 v}{\partial s^3}(x,y,\overline{\xi}_{n})ds.
\nonumber
\end{eqnarray*}
Thus,
\begin{eqnarray}
\label{eq15}
\lefteqn{\hspace{.5in}\left|\delta_t r_{q}\right|}
\\\nonumber\\  &\leq&
\left|\delta_t\frac{\partial^2 v}{\partial s^2}(x,y,t_{n})\right|\left|\frac{q^{3-\alpha}}{3-\alpha}
-\left[\frac{q^{2-\alpha}}{2}
+(q-1)^{2-\alpha}+(q-2)^{2-\alpha}+\cdots+1^{2-\alpha}\right]\right|\tau^{3-\alpha}
\nonumber\\\nonumber\\ &+&
\frac{1}{\tau}\left|\sum\limits_{k=n-q+1}^{n}\int_{t_{k-1}}^{t_k}\left\{g(s)
-\left[\frac{s-t_{k-1}}{\tau}g(t_{k})
+\frac{t_{k}-s}{\tau}g(t_{k-1})\right]\right\}(s-t_{n})\frac{\partial^3 v}{\partial s^3}(\cdot,\overline{\xi}_{n})ds\right.
\nonumber\\\nonumber\\ &-&
\left.\sum\limits_{k=n-q}^{n-1}\int_{t_{k-1}}^{t_k}\left\{g(s)
-\left[\frac{s-t_{k-1}}{\tau}g(t_{k})
+\frac{t_{k}-s}{\tau}g(t_{k-1})\right]\right\}(s-t_{n-1})\frac{\partial^3 v(\cdot,\overline{\xi}_{n-1})}{\partial s^3}ds\right|
\nonumber\\\nonumber\\ &\leq&
C\max\limits_{t_0\leq t\leq t_n}\left|\frac{\partial^3 v}{\partial t^3}(x,y,t)\right|\tau^{3-\alpha}.
\nonumber
\end{eqnarray}

By using the notation $v^{n-\frac{1}{2}}=\frac{1}{2}(v^n+v^{n-1})$, $\delta_t u^{n}=\frac{u^n-u^{n-1}}{\tau}$, and the Lemma 2.2 in \cite{sunzhizhong2012}, we have
\bea
\label{eqB3}
v^{n-\frac{1}{2}}=\delta_t u^{n}+(R_{t})^{n-\frac{1}{2}},\quad 1\leq n\leq M,
\eea
where
\begin{eqnarray}
\label{eqR-t}
\\&&
(R_t)^{n-\frac{1}{2}}=\frac{\tau^{2}}{16}\int_{0}^{1}\left[\frac{\partial^3 u}{\partial t^3}(x,y,t_{n-\frac{1}{2}}+\frac{s\tau}{2})+\frac{\partial^3 u}{\partial t^3}(x,y,t_{n-\frac{1}{2}}-\frac{s\tau}{2})\right](1-s^2) ds\nonumber.
\end{eqnarray}

Substituting (\ref{eqB3}) into (\ref{eqB1}), then
the truncation error in (\ref{eq:3-0}), $R_{\alpha}^{n-\frac{1}{2}}$, can be written as
\begin{eqnarray}
\label{eqR-alpha0}\\
\nonumber
R_{\alpha}^{n-\frac{1}{2}}&=&(\widetilde{R_t^{\alpha}})^{n-\frac{1}{2}}-\frac{\tau^{1-\alpha}}{\Gamma (3-\alpha)}\left[(R_{t})^{n-\frac{1}{2}}
-\sum\limits_{j=1}^{n-1}(b_{n-j-1}-b_{n-j})(R_{t})^{j-\frac{1}{2}}\right]
\\
\nonumber
\\
\nonumber
&=&\frac{1}{\Gamma (3-\alpha)}\left \{({R_t^{\alpha}})^{n-\frac{1}{2}}-\tau^{1-\alpha}\left [(R_{t})^{n-\frac{1}{2}}
-\sum\limits_{j=1}^{n-1}(b_{n-j-1}-b_{n-j})(R_{t})^{j-\frac{1}{2}}\right] \right \}.
\end{eqnarray}
Since $v=u_t$, then combining (\ref{eqR-alpha-t-1}), (\ref{eq19}) and (\ref{eq15}), we have
\begin{eqnarray}
\label{eq20}
\left|\delta_t(R_t^{\alpha})^{n}\right|
&\leq&
\left|\delta_tr_q\right|+\left|\delta_t\widehat{r}_{q}\right|
\\
\nonumber
&\leq&
C\tau^2t_{q-1}^{-\alpha}\max\limits_{t_0\leq t\leq t_n}\left|\frac{\partial^3 u}{\partial t^3}(x,y,t)\right|+
C\max\limits_{t_0\leq t\leq t_n}\left|\frac{\partial^4 u}{\partial t^4}(x,y,t)\right|\tau^{3-\alpha}.
\end{eqnarray}
For $ 2\leq n\leq M$, we have
\begin{eqnarray}
\label{eq5}
\delta_t R_{\alpha}^{n-\frac{1}{2}}
&=&\frac{1}{\Gamma (3-\alpha)}\left \{\delta_t(R_t^{\alpha})^{n-\frac{1}{2}}-\tau^{1-\alpha}\left[\delta_t(R_{t})^{n-\frac{1}{2}}\right.\right.
\\
&&\left .\left.
\hspace{.25in}-\sum\limits_{j=1}^{n-2}(b_{j-1}-b_{j})\delta_t(R_{t})^{n-j-\frac{1}{2}}-\frac{b_{n-2}-b_{n-1}}{\tau }(R_{t})^{\frac{1}{2}}\right]\right\},
\nonumber
\end{eqnarray}
Thus, for $n\geq 2$
\begin{eqnarray}
\label{eq6}
&&\left \|\delta_t R_{\alpha}^{n-\frac{1}{2}}\right\|_D\leq \frac{1}{\Gamma (3-\alpha)}\left\{\left \|\delta_t(R_t^{\alpha})^{n-\frac{1}{2}}\right\|_D+\tau^{1-\alpha}\left \|\delta_t(R_{t})^{n-\frac{1}{2}}\right\|_D\right.
\\ &&\left.
+\tau^{1-\alpha}\left[
\sum\limits_{j=1}^{n-2}(b_{j-1}-b_{j})\left \|\delta_t(R_{t})^{n-j-\frac{1}{2}}\right\|_D+\frac{b_{n-2}-b_{n-1}}{\tau }\left \|(R_{t})^{\frac{1}{2}}\right\|_D\right]\right\}.
\nonumber
\end{eqnarray}
From (\ref{eqR-t}), we have
\begin{eqnarray}
\label{eqR-t-1}
\left|(R_t)^{n-\frac{1}{2}}\right|\leq \frac{\tau^2}{12}\max\limits_{0\leq t \leq T}\left|\frac{\partial^3 u}{\partial t^3}(x,y,t)\right|.
\end{eqnarray}
Combining (\ref{eqR-t}) and (\ref{eqR-t-1}), we have
\begin{eqnarray}
\label{eq7}
\left|\delta_t(R_t)^{n-\frac{1}{2}}\right|\leq \frac{\tau^2}{12}\max\limits_{0\leq t \leq T}\left|\frac{\partial^4 u}{\partial t^4}(x,y,t)\right|.
\end{eqnarray}

Since $\sum\limits_{j=1}^{n-2}(b_{j-1}-b_{j})=1-b_{n-2}\leq1$, and $(2-\alpha)(j+1)^{1-\alpha}< b_j<(2-\alpha)j^{1-\alpha}$, we have
\begin{eqnarray*}
\label{eq8}
b_{n-2}-b_{n-1}&<&(2-\alpha)\left[(n-2)^{1-\alpha}-n^{1-\alpha}\right]
\\ &=&
(2-\alpha)(1-\alpha)\int_{n}^{n-2}x^{-\alpha}dx
\nonumber \\ &\leq&
2(2-\alpha)(\alpha-1)(n-2)^{-\alpha}.
\nonumber
\end{eqnarray*}
Thus
\begin{equation}
\label{eq21}
\tau^{1-\alpha}\left[
\sum\limits_{j=1}^{n-2}(b_{j-1}-b_{j})\left|\delta_t(R_{t})^{n-j-\frac{1}{2}}\right|\right]
\leq
\frac{\tau^{3-\alpha}}{12}\max\limits_{0\leq t \leq T}\left|\frac{\partial^4 u}{\partial t^4}(x,y,t)\right|,
\end{equation}
and
\begin{equation}
\label{eq22}
\frac{b_{n-2}-b_{n-1}}{\tau^{\alpha} }\left|(R_{t})^{\frac{1}{2}}\right|
\leq
\frac{\tau^{2}(2-\alpha)(\alpha-1)t_{n-2}^{-\alpha}}{6}\max\limits_{0\leq t \leq T}\left|\frac{\partial^4 u}{\partial t^4}(x,y,t)\right|
.
\end{equation}
Combining (\ref{eq20}), (\ref{eq5}),  (\ref{eq7}), (\ref{eq21}) and (\ref{eq22}), we obtain
\begin{eqnarray}
\label{eq23}
\left |\delta_t R_{\alpha}^{n-\frac{1}{2}}\right|&\leq& C\left\{\tau^2t_{q-1}^{-\alpha}\max\limits_{t_0\leq t\leq t_n}\left|\frac{\partial^3 u}{\partial t^3}(x,y,t)\right|+\tau^{3-\alpha}\max\limits_{0\leq t \leq T}\left|\frac{\partial^4 u}{\partial t^4}(x,y,t)\right|\right.
\\
&\quad&
\left.+\frac{\tau^{2}(2-\alpha)(\alpha-1)t_{n-2}^{-\alpha}}{6}\max\limits_{0\leq t \leq T}\left|\frac{\partial^4 u}{\partial t^4}(x,y,t)\right|\right\}.
\nonumber
\end{eqnarray}
Since $1<\alpha<2$, we have $2>3-\alpha$. Therefore, using the (\ref{eq23}),
we obtain
\[
\left \|\delta_t R_{\alpha}^{n-\frac{1}{2}}\right\|_D\;\leq\;
C\tau^{3-\alpha},
\]
as desired.


\begin{thebibliography}{10}

\bibitem{BaDiScTr}
{\sc D.~Baleanum, K.~Diethelm, E.~Scalas, and J.~J.~Trujillo},
{\em Fractional Calculus. Models and Numerical Methods}. Series on Complexity, Nonlinearity and Chaos, 3, World Scientific, New Jersey,
2012.

\bibitem{Bialecki4} {\sc B.~Bialecki}, {\em Convergence analysis of the orthogonal spline collocation for elliptic boundary value problems},
SIAM J. Numer. Anal., 35 (1998),
pp.~617--631.

\bibitem{BiFa}
{\sc B.~Bialecki and G.~Fairweather},
{\em Orthogonal spline collocation methods for partial differential equations},
J. Comput. Appl. Math.,
128(2001), 55--82.

\bibitem{Ca}
{\sc M. Caputo},
{\em Linear models of dissipation whose $Q$ is almost frequency independent.II}.
Reprinted from Geophys. J. R. Astr. Soc. 13 (1967), pp.~529--539. Fract. Calc. Appl. Anal., 11 (2008), pp.~4--14.

\bibitem{liufawangADI1}
{\sc S.~Chan and F.~Liu},
{\em ADI-Euler and extrapolation methods for the two-dimensional advection-dispersion equation},
J. Appl. Math. Comput., 26 (2008), pp.~295-–311.

\bibitem{cuimingrong1} {\sc M.~R.~Cui}, {\em Compact alternating direction implicit method for the two-dimensional time fractional diffusion equation}, J. Comput.
Phys., 231 (2012), pp.~2621--2633.

\bibitem{cuimingrong2} {\sc M.~R.~Cui}, {\em Convergence analysis of high-order compact alternating direction implicit schemes for the two-dimensional time fractional diffusion equation}, Numer. Algor., 62 (2013), pp.~383--409.

\bibitem{fairweather1978}{\sc G.~Fairweather},
 {\em Finite Element Galerkin Methods for Differential Equations},
 Lecture Notes
in Pure and Applied Mathematics, Volume 34, Marcel Dekker, New York, 1978.

\bibitem{FaGl}
{\sc G.~Fairweather and I.~Gladwell},
{\em Algorithms for almost block diagonal linear systems}, SIAM Rev.,
46(2004), 49--58.

\bibitem{FeBiFa}
{\sc R.~I.~Fernandes, B.~Bialecki, and G.~Fairweather},
{\em Alternating direction implicit orthogonal spline collocation methods for evolution equations},  in
Mathematical Modelling and Applications to Industrial Problems (MMIP-2011), M. J. Jacob and S. Panda, editors, Macmillan Publishers India Limited, 2012, pp. 3--11.

\bibitem{FeFa}{\sc R.~I.~Fernandes and G.~Fairweather},  {\em Analysis of alternating direction collocation methods for parabolic
and hyperbolic problems in two space variables}, Numer. Methods Partial Differ. Equ., 9
(1993), pp.~191--211.

\bibitem{FeFajcp}{\sc R.~I.~Fernandes and G.~Fairweather}, {\em An ADI extrapolated Crank-Nicolson orthogonal spline collocation method for nonlinear reaction-diffusion systems}, J. Comput. Phys., 231 (2012), pp.~6248--6267.

\bibitem{GaSu}
{\sc Z.~Z.~Sun and X.~N.~Wu}, {\em A fully discrete difference scheme for a diffusion-wave system}, Appl.
Numer. Math., 56 (2006), pp.~193--209.

\bibitem{JiLaZh}
{\sc B.~Jin, R.~Lazarov, and Z.~Zhiu},
{\em Error estimates for a semidiscrete finite element method for fractional order parabolic equations},
SIAM J. Numer. Anal., 51 (2013), pp.~445-466.

\bibitem{liufawangADI2}{\sc Q.~Liu, F.~Liu, I.~Turner, and V.~Anh}, {\em Numerical simulation for the 3D seepage flow with fractional derivatives in porous media}, IMA J. Appl.
Math., 74 (2009), pp.~201--229.

\bibitem{M3eScTa}
{\sc M.~M.~Meerschaert, H.~P.~Scheffler, and C.~Tadjeran},
{\em Finite difference methods for two-dimensional fractional dispersion equation},
J. Comput. Phys., 211 (2006), pp.~249--261.

\bibitem{MuMc}
{\sc K.~Mustapha and W.~McLean},
{\em Superconvergence of a discontinuous Galerkin method for fractional diffusion and wave equations},
SIAM J. Numer. Anal., 51 (2013), pp.~491--515.

\bibitem{OlSp}
{\sc K.~Oldham and J.~Spanier},
{\em The Fractional Calculus: Theory and Applications of Differentiation and Integration to Arbitrary Order},
Mathematics in Science and Engineering, vol. 111, Academic Press, New York, 1974.

\bibitem{fairweather2010}{\sc  A.~Pani, G.~Fairweather, and R.~I.~Fernandes}, {\em ADI orthogonal spline collocation methods for parabolic partial integro-differential equations}, IMA J. Numer. Anal., 30 (2010), pp.~248--276.

\bibitem{ADIqiyuan2}{\sc D.~W.~Peaceman and H.~H.~Rachford Jr.}, {\em The numerical solution of parabolic and elliptic differential equations}, J. Soc. Indust. Appl. Math., 3 (1955), pp.~28--41.

\bibitem{ReSu}
{\sc J.~Ren and Z.~Z.~Sun},
{\em Numerical algorithm with high spatial accuracy for the fractional diffusion-wave equation with Neumann boundary conditions},
J. Sci Comput., (2013), DOI 10.1007/s10915-012-9681-9.

\bibitem{cuiminghou32}{\sc C.~Tadjeran and M.~M.~Meerschaert}, {\em A second-order accurate numerical method for the twodimensional
fractional diffusion equation}, J. Comput. Phys. 220 (2007), pp.~813--823.

\bibitem{Uc1}
{\sc V.~V.~Uchaikin},
{\em Fractional Derivatives for Physicists and Engineers Vol. I, Background and Theory},
Higher Education Press, Beijing; Springer, Heidelberg, 2013.

\bibitem{Uc2}
{\sc V.~V.~Uchaikin},
{\em Fractional Derivatives for Physicists and Engineers Vol. II, Applications},
Higher Education Press, Beijing; Springer, Heidelberg, 2013.

\bibitem{cuiminghou35}{\sc H.~Wang and K.~Wang}, {\em An $O(Nlog^2N)$ alternating-direction finite difference method for two-dimensional
fractional diffusion equations}, J. Comput. Phys. 230 (2011), pp.~7830--7839.

\bibitem{liufawangADI3} {\sc Q.~Yu, F.~Liu, I.~Turner, and K.~Burrage},
{\em A computationally effective alternating direction method for the space
and time fractional Bloch-Torrey equation in 3-D}, Appl. Math. Comput., 219 (2012), pp.~4082--4095.

\bibitem{sunzhizhong7}{\sc Y.~N.~Zhang and Z.~Z.~Sun}, {\em Alternating direction implicit schemes for the two-dimensional
fractional sub-diffusion equation}, J. Comput. Phys., 230 (2011), pp.~8713--8728.

\bibitem{sunzhizhong2012}{\sc Y.~N.~Zhang, Z.~Z.~Sun, and X.~Zhao},
{\em Compact alternating direction implicit scheme for the two-dimensional fractional diffusion-wave equation}, SIAM J. Numer. Anal., 50 (2012), pp.~1535--1555.
\end{thebibliography}
\end{document}